\documentclass[12pt,a4paper]{article}

\usepackage[T1]{fontenc}
\usepackage{ae}	
\usepackage[italian,UKenglish]{babel}    
\usepackage{amsmath}
\usepackage{amsfonts}
\usepackage{amssymb}
\usepackage{amsthm}
\usepackage{algpseudocode}
\usepackage{float}
\usepackage{xcolor}
\usepackage{mathrsfs}
\usepackage{import}
\usepackage{geometry}

\usepackage{lettrine}   

\usepackage{verbatim}
\usepackage{alltt}

\usepackage{graphics}
\usepackage{graphicx}

\graphicspath{{./imgs/}}

\usepackage{subfigure}
\usepackage{floatflt}

\usepackage[font=small,labelfont=bf]{caption}

\theoremstyle{plain}
\newtheorem{theorem}{Theorem}[section]

\newtheorem{definition}[theorem]{Definition}

\newtheorem{lemma}{Lemma}[section]

\newtheorem{proposition}[theorem]{Proposition}
\newtheorem{remark}[theorem]{Remark}

\numberwithin{equation}{section}
\theoremstyle{definition}

\newcommand{\R}{\ensuremath{\mathbb{R}}}

\begin{document}

\title{Blow-up and global existence for solutions to the
porous medium equation with reaction and\\ slowly decaying density}
%
%
\author{Giulia Meglioli\thanks{Dipartimento di Matematica, Politecnico di Milano, Italia (giulia.meglioli@polimi.it).}\,\, and Fabio Punzo\thanks{Dipartimento di Matematica, Politecnico di Milano, Italia (fabio.punzo@polimi.it).}}

\date{}
%

%

\maketitle              

\begin{abstract}
We study existence of global solutions and finite time blow-up of solutions to the Cauchy problem for the porous medium equation with a variable density $\rho(x)$ and a power-like reaction term $\rho(x) u^p$ with $p>1$; this is a mathematical model of a thermal evolution of a heated plasma (see \cite{KR3}). The density decays slowly at infinity, in the sense that $\rho(x)\lesssim |x|^{-q}$ as $|x|\to +\infty$ with $q\in [0, 2).$
We show that for large enough initial data, solutions blow-up in finite time for any $p>1$. On the other hand, if the initial datum is small enough and $p>\bar p$, for a suitable $\bar p$ depending on $\rho, m, N$, then global solutions exist. In addition, if $p<\underline p$, for a suitable $\underline p\leq \bar p$ depending on $\rho, m, N$, then the solution blows-up in finite time for any nontrivial initial datum; we need the extra hypotehsis that $q\in [0, \epsilon)$ for $\epsilon>0$ small enough, when $m\leq p<\underline p$. Observe that $\underline p=\overline p$, if $\rho(x)$ is a multiple of $|x|^{-q}$ for $|x|$ large enough. Such results are in agreement with those established in \cite{SGKM}, where $\rho(x)\equiv 1$, and are related to some results in \cite{MT, MTS}. The case of fast decaying density at infinity, i.e. $q\geq 2$, is examined in \cite{MP2}.
\end{abstract}

\bigskip



\section{Introduction}
We investigate global existence and blow-up of nonnegative solutions to the Cauchy parabolic problem
\begin{equation}\label{problema}
\begin{cases}
\rho (x) u_t = \Delta(u^m)+\rho(x) u^p & \text{in }\, \R^N \times (0,\tau) \\
u=u_0& \text{in }\, \R^N \times \{0\}\,,
\end{cases}
\end{equation}
where $m>1, p>1$, $N\geq 3, \tau>0$; furthermore, we always assume that
\begin{equation}\tag{{\it $H$}} \label{hp}
\begin{cases}
\textrm{(i)} \; \rho\in C(\mathbb R^N),\, \rho>0\,\, \textrm{in}\,\, \mathbb R^N\,;\\
\textrm{(ii)} \; \textrm{there exist} \,\, k_1, k_2\in (0, +\infty)\,\, \textrm{with}\,\, k_1\leq k_2 \textrm{ and }\, 0\leq q<2 \,\, \textrm{such that}\\
\quad \;\;\,\, k_1|x|^q \le\dfrac{1}{\rho(x)}\le k_2|x|^q\quad \textrm{for all}\,\,\, x\in \mathbb R^N\setminus B_1(0)\,;\\
\textrm{(iii)}  \;  u_0\in L^\infty(\mathbb R^N), \,\,u_0\geq 0\,\, \textrm{in}\,\, \mathbb R^N\,.
\end{cases}
\end{equation}

The parabolic equation in problem \eqref{problema} is of the {\it porous medium} type, with a variable density $\rho(x)$ and a reaction term $\rho(x) u^p$.
Clearly, such parabolic equation is degenerate, since $m>1$. Moreover, the differential equation in \eqref{problema} is equivalent to
\[ u_t =\frac 1{\rho(x)} \Delta(u^m)+ u^p \qquad \text{in } \R^N \times (0,\tau);\]
therefore, the related diffusion operator is $\frac 1{\rho(x)}\Delta$, and in view of \eqref{hp}, the coefficient $\dfrac1{\rho(x)}$ can positively diverge at infinity.
Problem \eqref{problema} has been introduced in \cite{KR3} as a mathematical model of evolution of plasma temperature, where $u$ is the temperature, $\rho(x)$ is the particle density, $\rho(x) u^p$ represents the volumetric heating of plasma. Indeed, in \cite[Introduction]{KR3} a more general source term of the type $A(x) u^p$ has also been considered; however, then the authors assume that  $A\equiv 0$; only some remarks for the case $A(x)=\rho(x)$ are made in \cite[Section 4]{KR3}, when the problem is set in a slab in one space dimension. Then in \cite{KR1} and \cite{KR2} problem \eqref{problema} is dealt with in the case without the reaction term $\rho(x) u^p.$

\smallskip

We refer to $\rho(x)$ as a {\em slowly decaying density} at infinity because, in view of \eqref{hp},
\[\frac{1}{k_2|x|^q}\leq \rho(x)\leq \frac{1}{k_1|x|^q}\quad \textrm{for all }\,\,|x|>1\,,\]
with
$$0\leq q<2.$$
Global existence and blow-up of solutions for problem \eqref{problema} with {\it fast decaying density} at infinity, i.e. $q\geq 2$, is investigated in \cite{MP2}\,. We regard the value $q=2$ as the threshold one, indeed, the behavior of solutions is very different according to the fact that $q<2$ or $q=2$ or $q>2$. Such important role played by the value $q=2$ does not surprise. In fact, for problem \eqref{problema} without the reaction term $u^p$, that is \begin{equation}\label{problema3}
\begin{cases}
\rho u_t = \Delta(u^m) & \text{in }\, \R^N \times (0,\tau) \\
u=u_0 & \text{in }\, \R^N \times \{0\}\,,
\end{cases}
\end{equation}
in \cite{P1}, it is shown that for $q\leq2$ there exists a unique bounded solution, whereas for $q> 2$, for any $u_0\in L^\infty(\R^N)$ there exist infinitely many bounded solutions.

\smallskip

Let us briefly recall some results in the literature concerning well-posedness for problems related to \eqref{problema}. Problem \eqref{problema} with $\rho\equiv 1$ and without the reaction term, that is
\begin{equation}\label{problema2}
\begin{cases}
u_t = \Delta(u^m) & \text{in }\, \R^N \times (0,\tau) \\
u=u_0 & \text{in }\, \R^N \times \{0\},
\end{cases}
\end{equation}
has been the object of detailed investigations. We refer the reader to the book \cite{Vaz07} and references therein, for a comprehensive account of the main results. Also problem \eqref{problema} with variable density, without reaction term, that is problem \eqref{problema3}, has been widely examined. In particular, depending on the behaviour of $\rho(x)$ as $|x|\to\infty,$ existence and uniqueness of solutions and the asymptotic behaviour of solutions for large times have been addressed
(see, e.g., \cite{Eid90,EK, GMPor, GMPfra1, GMPfra2, I, I2, IM, KKT, KPT, KP1, KP2, KRV10, KR1, KR2, P1, RV06, RV08, RV09}).

\smallskip

For problem \eqref{problema} with $m=1$ and $\rho\equiv 1$, global existence and blow-up of solutions have been studied. To be specific, if
$$p\leq 1+\frac 2 N,$$ then finite time blow-up occurs, for all nontrivial nonnegative data, whereas, for
$$p>1+\frac 2 N,$$ global existence prevails for sufficiently small initial conditions (see, e.g., \cite{CFG, DL, FI, F, H, L, Q, Sacks, S, Y}). In addition, in \cite{LX} (see also \cite{dPRS}), problem \eqref{problema} with $m=1$ has been considered. Let assumption \eqref{hp} be satisfied, and let
\begin{equation}\label{eq140}
b:=2-q.
\end{equation}
Obviously, since $q\in [0, 2)$, we have that
\[b\in (0, 2]\,.\]
It is shown that if
$$p\leq 1+ \frac{b}{N-2+b},$$ then solutions blow-up in finite time, for all nontrivial nonnegative data, whereas, for
$$p>1+\frac{b}{N-2+b},$$ global in time solutions exist, provided that $u_0$ is small enough.

\smallskip

Now, let us recall some results established in \cite{SGKM} for problem \eqref{problema} with $\rho\equiv 1, m>1, p>1$ (see also \cite{GV, MQV}). We have:
\begin{itemize}
\item (\cite[Theorem 1, p. 216]{SGKM}) For any $p>1$, for all sufficiently large initial data, solutions blow-up in finite time;
\item (\cite[Theorem 2, p. 217]{SGKM}) if $p\in\left(1,m+\frac2N\right)$, for \it all \rm initial data, solutions blow-up in finite time;
\item (\cite[Theorem 3, p. 220]{SGKM}) if $p>m+\frac2N$, for all sufficiently small initial data, solutions exist globally in time.
\end{itemize}
Similar results for quasilinear parabolic equations, also involving $p$-Laplace type operators or double-nonlinear operators, have been stated in \cite{AfT}, \cite{A1},  \cite{AT1},  \cite{MT}, \cite{MTS1}, \cite{MT2},  \cite{MP}, \cite{MP2}, \cite{PT}, \cite{T1}, \cite{WZ} (see also \cite{MMP} for the case of Riemannian manifolds); moreover, in \cite{GMPhyp} the same problem on Cartan-Hadamard manifolds has been investigated.

Let us observe that the results in \cite{SGKM} illustrated above have been proved by means of comparison principles and suitable sub-- and supersolutions of the form
\[w(x,t)=C\zeta(t)\left [ 1- \frac{|x|^2}{a} \eta(t) \right ]_{+}^{\frac{1}{m-1}}\quad \textrm{for any}\,\,\, (x,t)\in \mathbb R^N\times [0, T),\]
for appropriate auxiliary functions $\zeta=\zeta(t), \eta=\eta(t)$ and constants $C>0, a>0$.

\smallskip

In \cite{MT, MTS} double-nonlinear operators, including in particular problem \eqref{problema}, are investigated. It is shown that (see \cite[Theorem 1]{MT}) if $\rho(x)=|x|^{-q}$ with $q\in (0,2)$, for any $x\in \mathbb R^N\setminus \{0\}$, 
\[p>m+\frac{b}{N-2+b},\]
$u_0\geq 0$ and
\begin{equation}\label{eqt1}
\int_{\mathbb R^N}\left\{u_0(x) + [u_0(x)]^{\bar q} \right\}\rho(x) dx<\delta,
\end{equation}
for some $\delta>0$ small enough and $\bar q>\frac N2 (p-m)$, then there exists a global solution of problem \eqref{problema}. In addition, a smoothing estimate holds. On the other hand, if $\rho(x)=|x|^{-q}$
 or $\rho(x)=(1+|x|)^{-q}$ with $q\in [0,2)$, $u_0\not\equiv 0$ and 
 \[p<m+\frac{b}{N-2+b},\] then blow-up prevails, in the sense that there exist $\theta\in (0,1), R>0, T>0$ such that 
 \[\int_{B_R}[u(x,t)]^{\theta} \rho(x) dx\to +\infty \quad \textrm{as}\,\,\, t\to T^-.\]
Such results have also been generalized to more general initial data, decaying at infinity with a certain rate (see \cite{MTS}). We compare the results in \cite{MT} with ours below (see Remarks \ref{ossMT1}, \ref{ossMT2} and \ref{ossMT3}).
 
\subsection{Outline of our results} We prove the following results. 
\begin{itemize}
\item (See Theorem \ref{teosupersolution}). Suppose that
\begin{equation}\label{eq2}
\frac{k_2}{k_1} < m+ \frac{(m-1)(N-2)}{b}\,,
\end{equation}
and define
\begin{equation}\label{eq1}
\overline p:=\frac{m(N-2+b)+ \frac{b}{m-1}(m-\frac{k_2}{k_1})}{N-2+\frac{b}{m-1}\left(m-\frac{k_2}{k_1}\right)}\,.
\end{equation}
If $u_0$ has compact support and is small enough, 
\[p>\overline p,\]
then global solutions exist.

Note that for $k_1=k_2$,
\[\overline p= m+\frac{b}{N+2-b};\] 
this is coherent with \cite[Theorem 1]{MT} (see Remark \ref{ossMT1} below for more details). 
If in addition $\rho\equiv 1$, and so $b=2$, we have
\[\overline p=m+\frac 2 N\,.\]
Thus, our results are in accordance with those in \cite{SGKM}. Furthermore, for $m=1$, they are in agreement with the results established in \cite{LX}, and in \cite{F, H} when $\rho\equiv 1$.

\item (See Theorem \ref{teosubsolution}). For any $p>1$, if $u_0$ is sufficiently large, then solutions to problem \eqref{problema} blow-up in finite time.
\item (see Theorem \ref{teo4}). If $1<p<m$, then for any $u_0\not\equiv 0$, solutions to problem \eqref{problema} blow-up in finite time. In addition (see Theorem \ref{teo3}), if $$m\leq  p<m+\frac{b}{N-2+b}$$ and $q\in[0,\epsilon)$ for $\epsilon>0$ small enough, then for any $u_0\not\equiv 0$, solutions to problem \eqref{problema} blow-up in finite time.
\end{itemize}
It remains to be understood if the restriction $q\in [0, \epsilon)$ can be removed.

\smallskip

Actually, we obtain similar results to those described above, also when assumption \eqref{hp} is fulfilled for general $0<k_1<k_2$. In that case, the blow-up result for large initial data can be stated exactly as in the previous case $k_1=k_2$ . Instead, in order to get global existence, the assumption on $p$ changes, since it also depends on the parameters $k_1$ and $k_2$. More precisely,
 Indeed, also our blow-up results for any nontrivial initial datum holds when $0<k_1<k_2$. The case $1<p<m$ is exactly as before. Moreover (see Theorem \ref{teo3}), if
\[m\leq p< \underline p,\]
where
\begin{equation}\label{psotto}
\underline p=\dfrac{m\left(N-2+b\right)+\frac{b}{m-1}\left(m-\frac{k_1}{k_2}\right)}{N-2+\frac{b}{m-1}\left(m-\frac{k_1}{k_2}\right)}\,,
\end{equation}
then the solution blows-up for any nontrivial initial datum, under the extra hypothesis that $q\in [0, \epsilon)$ for $\epsilon>0$ small enough. Note that in view of \eqref{eq2}, it can be easily checked that
$$\underline p \leq\bar p\,.$$
In particular, $\underline p=\bar p$ whenever $k_1=k_2.$
\smallskip

The methods used in \cite{dPRS, F, H, LX}  cannot work in the present situation, since they strongly require $m=1$.
Indeed, our proofs mainly relies on suitable comparison principles (see Propositions \ref{cpsup}, \ref{cpsub}) and properly constructed sub- and supersolutions.
Let us mention that the arguments exploited in \cite{SGKM} cannot be directly used in our case, due to the presence of the coefficient $\rho(x)$. In fact, we construct appropriate sub-- and supersolutions, which crucially depend on the behavior at infinity of the inhomogeneity term $\rho(x)$. More precisely, whenever $|x|>1$, they are of the type
\[w(x,t)=C\zeta(t)\left [ 1- \frac{|x|^b}{a} \eta(t) \right ]_{+}^{\frac{1}{m-1}}\quad \textrm{for any}\,\,\, (x,t)\in \big[\mathbb R^N\setminus B_1(0)\big]\times [0, T),\]
for suitable functions $\zeta=\zeta(t), \eta=\eta(t)$ and constants $C>0, a>0$.
In view of the term $|x|^b$ with $b\in (0,2]$, we cannot show that such functions are sub- and supersolutions in $B_1(0)\times (0, T)$. Thus we have to extend them in a suitable way in $B_1(0)\times (0, T)$. This is not only a technical aspect. In fact, in order to extend our sub-- and supersolutions, we need to impose some extra conditions on $\zeta=\zeta(t)$, $\eta=\eta(t)$, $C$ and $a.$ Thus, it appears a sort of interplay between the behavior of the density $\rho(x)$ in compact sets, say $B_1(0)$, and its behavior for large values of $|x|$. Finally, let us comment about the proofs of the blow-up result for any nontrivial initial datum. For $1<p<m$, the result follows by a direct application of Theorem \ref{teosubsolution}. For $m<p<\underline p, $ the proof is more involved. The corresponding result for the case $\rho\equiv 1$ established in \cite{SGKM} is proved by means of the Barenblatt solutions of the porous medium equation
\[u_t = \Delta u^m \quad \textrm{in}\,\, \mathbb R^N\times (0, +\infty)\,.\]
In our situation, we do not have self-similar solutions, since our equation in \eqref{problema} is not scaling invariant, in view of the presence of the term $\rho(x)$. Indeed, we construct a suitable subsolution $z$ of equation
\[u_t = \frac 1{\rho} \Delta u^m \quad \textrm{in}\,\, \mathbb R^N\times (0, +\infty)\,.\]
By means of $z$, we can show that after a certain time, the solution $u$ of problem \eqref{problema} satisfies the hypotheses required by Theorem \ref{teosubsolution}. Hence $u$ blows-up in finite time.

\medskip

The paper is organized as follows. In Section \ref{statements} we state our main results, in Section \ref{prel} we give the precise definitions of solutions, we establish a local in time existence result and some useful comparison principles. In Section \ref{gepr} we prove the global existence theorem. The blow-up results are proved in Section \ref{bupr} for sufficiently big initial data, and in Section \ref{bu} for any initial datum.

\section{Statements of the main results}\label{statements}

In view of \eqref{hp}-(i), there exist $\rho_1, \rho_2\in (0, +\infty)$ with $\rho_1\leq \rho_2$ such that
\begin{equation}
\rho_1\leq \frac1{\rho(x)}\leq \rho_2 \quad \textrm{for all}\,\,\, x\in \overline{B_1(0)}.
\label{hyprho}
\end{equation}
As a consequence of hypothesis \eqref{hp} and \eqref{hyprho}, we can assume that
\begin{equation}\label{eq30b}
k_1=\rho_1, \quad k_2=\rho_2\,.
\end{equation}
Let $\overline p$ be defined by \eqref{eq1}. It is immediate to see that $\overline p$ is monotonically increasing with respect to the ratio $\frac{k_2}{k_1}$; furthermore,
\[\overline p>m\,.\]

\smallskip

Define
\begin{equation}\label{eq141}
\mathfrak{r}(x):=\begin{cases}
|x|^b & \textrm{if}\quad |x|\geq 1,\\
\frac{b|x|^b+2-b}{b} & \textrm{if} \quad |x|<1\,.
\end{cases}
\end{equation}

The first result concerns the global existence of solutions to problem \eqref{problema} for $p>\overline p$.

\begin{theorem}\label{teosupersolution}
Let assumptions \eqref{hp},  \eqref{eq2} and \eqref{eq30b} be satisfied. Suppose that $$p>\overline p,$$ where $\overline p$ is given in \eqref{eq1}, and that $u_0$ is small enough and has compact support. Then problem \eqref{problema} admits a global solution $u\in L^\infty(\mathbb R^N\times (0, +\infty))$.

More precisely, if $C>0$ is small enough, $T>0$ is big enough, $a>0$ with $$\omega_0\leq \frac{C^{m-1}}{a}\leq \omega_1,$$ for suitable $0<\omega_0<\omega_1$,
\begin{equation}\label{hpalpha}
\alpha\in \left(\frac{1}{p-1}, \frac{1}{m-1}\right), \quad \beta=1-\alpha (m-1),
\end{equation}
\begin{equation}\label{eq6}
u_0(x) \le CT^{-\alpha} \left [1-\frac{\mathfrak{r}(x)}{a}T^{-\beta} \right ]^{\frac{1}{m-1}}_{+}\quad \textrm{for any}\,\,\, x\in \mathbb R^N\,,
\end{equation}
then problem \eqref{problema} admits a global solution $u\in L^\infty(\mathbb R^N\times (0, +\infty))$. Moreover,
\begin{equation}\label{eq143}
u(x,t) \le C(T+t)^{-\alpha} \left [ 1- \frac{\mathfrak{r}(x)}{a} (T+t)^{-\beta} \right ]_{+}^{\frac{1}{m-1}}\quad \textrm{for any}\,\,\, (x,t)\in \mathbb R^N\times [0, +\infty)\,.
\end{equation}
\end{theorem}

The precise choice of the parameters $C>0, T>0$ and $a>0$ in Theorem \ref{teosupersolution} is discussed in Remark \ref{ParThmGlob} below. Observe that if $u_0$ satisfies \eqref{eq6}, then
\[\|u_0\|_{\infty}\leq C T^{-\alpha},\]
\[\operatorname{supp}u_0\subseteq \{x\in \mathbb R^N\,:\, \mathfrak{r}(x)\leq a T^{\beta}\}\,.\]
In view of the choice of $C, T, a$ (see also Remark \ref{ParThmGlob}), $\|u_0\|_{\infty}$ is small enough, but $\operatorname{supp}\, u_0$ can be large, since we can select $a T^\beta>r_0$ for any fixed $r_0>0$.

Moreover, from \eqref{eq143} we can infer that
\begin{equation}\label{eq150}
\operatorname{supp}u(\cdot, t)\subseteq \{x\in \mathbb R^N\,:\, \mathfrak{r}(x)\leq a (T+t)^{\beta}\}\quad \textrm{for all } t>0\,.
\end{equation}

\medskip

\begin{remark}
Note that if $k_1=k_2$, then $$\overline p=m+\frac{b}{N-2+b}\,.$$
In particular, for $q=0$, i.e. $b=2$, we obtain
\[\overline p=m + \frac 2 N\,.\]
Hence, Theorem \ref{teosupersolution} is coherent with  the results in \cite{SGKM}.
\end{remark}

\begin{remark}\label{ossMT1} 
In \cite[Theorem 1]{MT} a similar global existence result is proved, for $\rho(x)=|x|^{-q}$ for any $x\in \mathbb R^N\setminus\{0\}$ with $q\in [0, 2)$ and for suitable $u_0$ not necessarily compactly supported. Clearly, such $\rho$ does not satisfy assumption \eqref{hp}. Moreover, we can consider a more general behaviour of $\rho(x)$ for $|x|$ large; this affects the definition of $\bar p$, and consequently the choice of $p$.  
The smallness condition in Theorem \ref{teosupersolution} is different from that in \cite{MT}, and it is not possible in general to say which is stronger. Moreover, since we consider $u_0$ with compact support, we can obtain the estimates \eqref{eq143} and \eqref{eq150}, which do not have a counterpart in \cite{MT}. Finally, in \cite{MT} energy methods are used and a smoothing estimate is derived; hence the proof is completely different from our.
\end{remark}

The next result concerns the blow-up of solutions in finite time, for every $p>1$ and $m>1$, provided that the initial datum is sufficiently large.

Let
\[\mathfrak{s}(x):=\begin{cases}
|x|^b & \textrm{if}\quad |x|>1, \\
|x|^2 & \textrm{if}\quad |x|\leq 1\,.
\end{cases}\]

\begin{theorem}\label{teosubsolution}
Let assumptions \eqref{hp} and \eqref{eq30b} be satisfied. For any $p>1, m>1$ and for any $T>0$, if the initial datum $u_0$ is large enough, then the solution $u$ of problem \eqref{problema} blows-up in a finite time $S\in (0,  T]$, in the sense that
\begin{equation}\label{eq146}
\|u(t)\|_{\infty} \to+ \infty \quad \textrm{as} \,\,\, t \to S^{-}.
\end{equation}
More precisely, we have the following three cases.
\begin{itemize}
\item[(a)] Let $p>m$. If $C>0, a>0$ are large enough, $T>0$,
\begin{equation}\label{eq147}
u_0(x)\geq C T^{-\frac{1}{p-1}} \left[1- \frac{\mathfrak{s}(x)}{a} T^{\frac{m-p}{p-1}}\right]_{+}^{\frac{1}{m-1}}\,,
\end{equation}
then the solution $u$ of problem \eqref{problema} blows-up and satisfies the bound from below
\begin{equation}\label{eq18}
u(x,t) \ge C (T-t)^{-\frac{1}{p-1}}\left[1- \frac{\mathfrak{s}(x)}{a} (T-t)^{\frac{m-p}{p-1}}\right]_{+}^{\frac{1}{m-1}} \,\, \textrm{for any} \,\,\,(x,t) \in \mathbb R^N\times [0,S)\,.
\end{equation}

\item[(b)] Let $p<m$. If $\frac{C^{m-1}}{a}>0$ and $a>0$ are big enough, $T>0$ and \eqref{eq147} holds, then
the solution $u$ of problem \eqref{problema} blows-up and satisfies the bound from below \eqref{eq18}.

\item[(c)] Let $p=m$. If $\frac{C^{m-1}}{a}>0$ and $a>0$ are big enough, $T>0$ and \eqref{eq147} holds, then
the solution $u$ of problem \eqref{problema} blows-up and satisfies the bound from below \eqref{eq18}.
\end{itemize}
\end{theorem}
Observe that if $u_0$ satisfies \eqref{eq147}, then
\[\operatorname{supp}u_0\supseteq \{x\in \mathbb R^N\,:\, \mathfrak{s}(x)< a T^{\frac{p-m}{p-1}}\}\,.\]
In all the cases $(a), (b), (c)$, from \eqref{eq18} we can infer that
\begin{equation}\label{eq151}
\operatorname{supp}u(\cdot, t)\supseteq \{x\in \mathbb R^N\,:\, \mathfrak{s}(x)<a(T-t)^{\frac{p-m}{p-1}}\}\quad \textrm{for all } t\in [0, S)\,.
\end{equation}
The precise choice of parameters $C>0$, $T>0$, $a>0$ in Theorem \ref{teosubsolution} is discussed in Remark \ref{PartThmBlo} below.

\begin{remark}\label{ossMT2}
Let us mention that in \cite{MT}, where some blow-up results are shown for problem \eqref{problema}, there is not a counterpart of Theorem \ref{teosubsolution}, since our result concerns {\em any} $p>1$ and sufficiently large initial data.
\end{remark}
\normalcolor

\subsection{Blow-up for any nontrivial  initial datum}
In this Subsection we discuss a further result concerning the blow-up of the solution to problem \eqref{problema} for any initial datum $u_0\in C(\mathbb R^N), u_0\geq 0, u_0\not\equiv 0$.

Let $\underline p$ and be defined by \eqref{psotto} and \eqref{eq1}, respectively. Assume \eqref{eq2}.  It is direct to see that
\begin{equation}\label{tildep}
\underline p \leq \bar p\,.
\end{equation}
In particular, $\underline p=\bar p$, whenever $k_1=k_2.$
We distinguish between two cases:
\begin{itemize}
\item[1)] $1<p < m$\,,
\item[2)] $m\leq p<\underline p$\,.
\end{itemize}

In case 2), we need an extra hypothesis. In fact, we assume that \eqref{hp} holds with
\begin{equation}\label{eq132}
q\in (0,\epsilon)\,,
\end{equation}
for some $\epsilon>0$ to be fixed small enough later. Then, $b$ defined by \eqref{eq140}, satisfies
\begin{equation}\label{beps}
2-\epsilon\,\,<\,\,b\,\,<\,\,2\,.
\end{equation}


\begin{theorem}\label{teo4}
Let assumption \eqref{hp} be satisfied.  Suppose that $$1<p< m\,,$$ and that $u_0\in C(\mathbb R^N), u_0(x) \not\equiv 0$. Then, for any sufficiently large $T>0$, the solution $u$ of problem \eqref{problema} blows-up in a finite time $S\in (0,  T]$, in the sense that $$\|u(t)\|_{\infty} \to+ \infty \quad \textrm{as} \,\,\, t \to S^{-}.$$ More precisely, the bound from below \eqref{eq18} holds, with $b, C, a, \zeta, \eta$ as in Theorem \ref{teosubsolution}-(b)\,.
\end{theorem}

\bigskip


\begin{theorem}\label{teo3}
Let assumptions $\eqref{hp}$ and \eqref{eq132} be satisfied for $\epsilon>0$ small enough. Let $u_0 \in C^{\infty}(\R^N)$ and $u_0 \not\equiv 0$. If
\begin{equation}\label{eq17}
m\leq p<\underline p,
\end{equation}
then there exist sufficiently large $t_1>0$ and $T>0$ such that the solution $u$ of problem \eqref{problema} blows-up in a finite time $S\in (0, T+t_1]$, in the sense that
$$\|u(t)\|_{\infty} \to+ \infty \quad \textrm{as}\,\,\, t \to S^{-}.$$
More precisely, when $S>t_1$, we have the bound from below
\begin{equation}\label{eq300ag}
u(x,t) \ge  C (T+t_1-t)^{-\frac{1}{p-1}}\left[1- \frac{\mathfrak{s}(x)}{a} (T+t_1-t)^{\frac{m-p}{p-1}}\right]_{+}^{\frac{1}{m-1}} \quad \textrm{for any} \,\,\,(x,t) \in \mathbb R^N\times (t_1, S)\,,
\end{equation}
with $C, a$ as in Theorem \ref{teosubsolution}-(a).
\end{theorem}

\begin{remark}\label{ossMT3}
As it has been mentioned in the Introduction, in \cite[Theorem 3]{MT}  blow-up of solutions to problem \eqref{problema} is shown when $\rho(x)=|x|^{-q}$ or $\rho(x)=(1+|x|)^{-q}$ with $q\in [0, 2)$. However, the results in \cite{MT} are different, in fact it is obtained an integral blow-up, that is, for some $R>0$, $\theta\in (0,1)$, $T>0$, $\int_{B_R} [u(x,t)^{\theta}]\rho(x) dx\to +\infty$ as $t\to T^-$. On the other hand, we should mention that the extra hypothesis \eqref{eq132}, that we need in Theorem \ref{teo3}, in \cite{MT} is not used. Furthermore, the methods of proofs in \cite{MT} are completely different, since they are based on the choice of a special test function and integration by parts.
\end{remark}

\section{Preliminaries}\label{prel}
In this section we give the precise definitions of solution of all problems we address, then we state a local in time existence result for problem \eqref{problema}. Moreover, we recall some useful comparison principles.

\smallskip

Throughout the paper we deal with {\em very weak} solutions to problem \eqref{problema} and to the same problem set in different domains, according to the following definitions.

\begin{definition}
Let $u_0\in L^{\infty}(\R^N)$ with $u_0\ge0$. Let $\tau>0$, $p>1, m>1$. We say that a nonnegative function $u\in L^{\infty}(\R^N\times (0,S))$ for any $ S<\tau$ is a solution of problem \eqref{problema} if
\begin{equation}
\begin{aligned}
-\int_{\R^N}^{}\int_{0}^{\tau} \rho(x) u \varphi_t \,dt\,dx &= \int_{\R^N} \rho(x) u_0(x) \varphi(x,0) \,dx \\ &+ \int_{\R^N}^{}\int_{0}^{\tau}  u^m \Delta \varphi \,dt\,dx \\ &+ \int_{\R^N}^{}\int_{0}^{\tau} \rho(x) u^p \varphi \,dt\,dx
\end{aligned}
\label{veryweak}
\end{equation}
for any $\varphi \in C_c^{\infty}(\R^N \times [0,\tau)), \varphi \ge 0.$ Moreover, we say that a nonnegative function $u\in L^{\infty}(\R^N\times (0,S))$ for any $ S<\tau$ is a subsolution (supersolution) if it satisfies \eqref{veryweak} with the inequality $"\le"$ ($"\ge"$) instead of $"="$ with $\varphi\geq 0$.
\label{soluzioneveryweak}
\end{definition}

For any $x_0\in \mathbb R^N$ and $R>0$ we set
$$B_R(x_0)=\{x\in \R^N :  \| x-x_0 \| < R \}.$$
When $x_0=0$, we write $B_R\equiv B_R(0).$ For every $ R>0$, we consider the auxiliary problem
\begin{equation}
\begin{cases}
 u_t=\frac 1{\rho(x)}\Delta(u^m) +u^p & \text{ in } B_R\times(0,\tau) \\
u=0  &  \text{ on } \partial B_R \times(0,\tau)\\
u=u_0 &  \text{ in } B_R\times\{0\}\,.\\
\end{cases}
\label{problemalocale}
\end{equation}

\begin{definition}
Let $u_0\in L^{\infty}(B_R)$ with $u_0\ge0$. Let $\tau>0$, $p>1, m>1$. We say that a nonnegative function $u\in L^{\infty}(B_R\times (0,S))$ for any $S<\tau$ is a solution of problem \eqref{problemalocale} if
\begin{equation}
\begin{aligned}
-\int_{B_R}^{}\int_{0}^{\tau} \rho(x) u\, \varphi_t \,dt\,dx &= \int_{B_R} \rho(x) u_0(x) \varphi(x,0) \,dx \\ &+ \int_{B_R}^{}\int_{0}^{\tau} u^m \Delta \varphi \,dt\,dx \\ &+ \int_{B_R}^{}\int_{0}^{\tau} \rho(x) u^p \varphi \,dt\,dx
\end{aligned}
\label{veryweaklocale}
\end{equation}
for any $\varphi \in C_c^{\infty}(\overline{B_R} \times [0,\tau))$ with $\varphi| _{\partial B_R}=0$ for all $t\in [0,\tau)$. Moreover, we say that a nonnegative function $u\in L^{\infty}(B_R\times (0,S))$ for any $S<\tau$ is a subsolution (supersolution) if it satisfies \eqref{veryweaklocale} with the inequality $"\le"$ ($"\ge"$) instead of $"="$, with $\varphi\geq 0$.
\label{soluzioneveryweaklocale}
\end{definition}

\begin{proposition}\label{exiloc}
Let hypothesis \eqref{hp} be satisfied. Then there exists a solution $u$ to problem \eqref{problemalocale} with
$$\tau\geq \tau_R:=\frac{1}{(p-1)\|u_0\|_{L^\infty(B_R)}^{p-1}}.$$
\end{proposition}
\begin{proof}
Note that $\underline u\equiv 0$ is a subsolution to \eqref{problemalocale}. Moreover, let $\bar u_R(t)$ be the solution of the Cauchy problem
$$
\begin{cases}
\bar u'(t)=\bar{u}^p \\
\bar{u}(0)=\| u_0 \|_{L^\infty(B_R)}\,,
\end{cases}
$$
that is
\[
\bar{u}_R(t)=\frac{\| u_0 \|_{L^\infty(B_R)}}{\left [1-(p-1)t\| u_0 \|_{L^\infty(B_R)}^{p-1} \right]^{\frac{1}{p-1}}}\quad \textrm{for all}\,\,\, t\in [0, \tau_R)\,.
\]
Clearly, for every $R>0$, $\bar u_R$ is a supersolution of problem \eqref{problemalocale}. Due to hypothesis \eqref{hp},
$$0<\min_{\bar B_R}\frac 1{\rho}\le\frac 1{\rho(x)}\le \max_{\bar B_R}\frac1{\rho}\quad \textrm{for all } x\in \overline{B_R}\,.$$
Hence, by standard results (see, e.g., \cite{Vaz07}), problem \eqref{problemalocale} admits a nonnegative solution $u_R\in L^{\infty}(B_R\times(0,S))$ for any $ S<\tau$, where $\tau\geq \tau_R$ is the maximal time of existence, i.e.
$$\| u_R(t)\|_{\infty} \to \infty \quad \textrm{as } \, t\to \tau_R^-.$$
\end{proof}

Moreover, the following comparison principle for problem \eqref{problemalocale} holds (see \cite{ACP} for the proof).
\begin{proposition}
Let assumption \eqref{hp} hold.
If $u$ is a subsolution of problem \eqref{problemalocale} and $v$ is a supersolution of \eqref{problemalocale}, then
$$u\le v \quad  \textrm{a.e. in } \, B_R \times (0,\tau).$$
\label{confrontolocale}
\end{proposition}

\begin{proposition}
Let hypothesis \eqref{hp} be satisfied. Then there exists a solution $u$ to problem \eqref{problema} with
$$\tau\geq \tau_0:=\frac{1}{(p-1)\|u_0\|_{\infty}^{p-1}}.$$
Moreover, $u$ is the {\em minimal solution}, in the sense that for any solution $v$ to problem \eqref{problema} there holds
\[u\leq v \quad \textrm{in }\,\,\, \mathbb R^N\times (0, \tau)\,.\]
\label{prop1}
\end{proposition}
\begin{proof}
For every $R>0$ let $u_R$ be the unique solution of problem \eqref{problemalocale}. It is easily seen that if $0<R_1<R_2$, then
\begin{equation}\label{eq170}
u_{R_1}\leq u_{R_2}\quad \textrm{in }\,\, B_{R_1}\times (0, \tau_0)\,.
\end{equation}
In fact, $u_{R_2}$ is a supersolution, while $u_{R_1}$ is a solution of problem \eqref{problemalocale} with $R=R_1$. Hence, by Proposition \ref{confrontolocale}, \eqref{eq170} follows. Let $\bar u(t)$ be the solution of
$$
\begin{cases}
\bar u'(t)=\bar{u}^p \\
\bar{u}(0)=\| u_0 \|_{\infty}\,,
\end{cases}
$$
that is
\[
\bar{u}(t)=\frac{\| u_0 \|_{\infty}}{\left [1-(p-1)t\| u_0 \|_{\infty}^{p-1} \right]^{\frac{1}{p-1}}}\quad \textrm{for all}\,\,\, t\in [0, \tau_0)\,.
\]
Clearly, for every $R>0$, $\bar u$ is a supersolution of problem \eqref{problemalocale}. Hence
\begin{equation}\label{eq171}
0\le u_R(x,t) \le \bar{u}\quad \textrm{in }\,\, B_{R}\times (0, \tau_0)\,.
\end{equation}
In view of \eqref{eq170}, the family $\{ u_R \}_{R>0}$ is monotone increasing w.r.t. $R$. Moreover, \eqref{eq171} implies that the family $\{u_R\}$ is uniformly bounded. Hence $\{ u_R \}_{R>0}$ converges point-wise to a function, say $u(x,t)$, as $R\to +\infty$, i.e.
$$
\lim_{R \to+ \infty}u_R(x,t) = u(x,t) \quad \textrm{a.e. in }\, \R^N \times (0,\tau_0)\,.
$$
Moreover, by the monotone convergence theorem, passing to the limit as $R\to +\infty$ in \eqref{veryweaklocale} we obtain
$$
\begin{aligned}
-\int_{\R^N}^{}\int_{0}^{\tau_0} \rho(x) u \varphi_t \,dt\,dx &= \int_{\R^N} \rho(x) u_0(x) \varphi(x,0) \,dx \\ &+ \int_{\R^N}^{}\int_{0}^{\tau_0}  u^m \Delta \varphi \,dt\,dx \\ &+ \int_{\R^N}^{}\int_{0}^{\tau_0} \rho(x) u^p \varphi \,dt\,dx
\end{aligned}
$$
for any $\varphi \in C_c^{\infty}(\R^N \times [0,\tau_0)), \varphi \ge 0.$ Hence $u$ is a solution of problem \eqref{problema} $u\in L^{\infty}(\mathbb R^N\times(0,S))$ for any $ S<\tau$, where $\tau\geq \tau_0$ is the maximal time of existence, i.e.
$$\| u(t)\|_{\infty} \to \infty \quad \textrm{as } \, t\to \tau^-.$$

Let us now prove that $u$ is the minimal nonnegative solution to problem \eqref{problema}. Let $v$ be any other solution to problem \eqref{problema}. Note that, for every $R>0$, $v$ is a supersolution to problem \eqref{problemalocale}. Hence, thanks to Proposition \ref{confrontolocale},
$$u_R\le v \quad  \text{ in } \,\, B_R \times (0,\tau).$$
Then passing to the limit as $R\to \infty$, we get
$$u\le v \text{ in } \, \R^N \times (0,\tau)\,.$$
Therefore, $u$ is the minimal nonnegative solution.
\end{proof}

In conclusion, we can state the following two comparison results.

\begin{proposition}\label{cpsup}
Let hypothesis \eqref{hp} be satisfied. Let $\bar{u}$ be a supersolution to problem \eqref{problema}. Then, if $u$ is the minimal solution to problem \eqref{problema} given by Proposition \ref{prop1}, then
\begin{equation}\label{eq172}
u\le\bar{u} \quad \text{a.e. in } \R^N \times (0,\tau)\,.
\end{equation}
In particular, if $\bar{u}$ exists until time $\tau$, then also $u$ exists at least until time $\tau$.
\end{proposition}
\begin{proof} Clearly, for any $R>0$, $\bar u$ is a supersolution to problem \ref{problemalocale}. Hence, by Proposition  \ref{confrontolocale},
\[u_R \leq \bar u\quad \textrm{in }\,\, B_R\times (0, \tau)\,.\]
By passing to the limit as $R\to +\infty$, we easily obtain \eqref{eq172}, which trivially ensures that $u$ does exist at least up to $\tau$, by the definition of maximal existence time.
\end{proof}

\begin{proposition}\label{cpsub}
Let hypothesis \eqref{hp} be satisfied. Let $u$ be a solution to problem \eqref{problema} for some time $\tau=\tau_1>0$ and $\underline{u}$ a subsolution to problem \eqref{problema} for some time $\tau=\tau_2>0$. Suppose also that
$$
\operatorname{supp }\underline{u}|_{\R^N\times[0,S]} \text{ is compact for every }  \, S\in (0, \tau_2)\,.
$$
Then
\begin{equation}\label{eq173}
u\ge\underline{u} \quad \text{ in }\,\, \R^N \times \left(0,\min\{\tau_1,\tau_2\}\right)\,.
\end{equation}
\end{proposition}
\begin{proof}
We fix any $S< \min\{\tau_1, \tau_2\}$. It $R>0$ is so large that $$\operatorname{supp }\underline{u}|_{\R^N\times[0,S]}\subseteq B_R\times [0, S],$$
then $u$ and $\underline u$ are a supersolution and a subsolution, respectively, to \ref{problemalocale}. Hence
$$u \geq \underline u \quad \textrm{in }\,\, B_R \times (0, S)\,.$$
Inequality \eqref{eq173} then just follows by letting $R\to+\infty$ and using the arbitrariness of $S$.
\end{proof}

\begin{remark}
Note that by minor modifications in the proof of  \cite[Theorem ]{Pu1} one could show that problem \eqref{problema} admits at most one bounded solution.
\end{remark}

In what follows we also consider solutions of equations of the form
\begin{equation}\label{eq189}
u_t = \frac 1{\rho(x)}\Delta(u^m) + u^p \quad \textrm{in }\,\, \Omega\times (0, \tau),
\end{equation}
where $\Omega\subseteq\mathbb R^N$. Solutions are meant in the following sense.

 \begin{definition}\label{soldom}
Let $\tau>0$, $p>1, m>1$. We say that a nonnegative function $u\in L^{\infty}(\Omega\times (0,S))$ for any $S<\tau$ is a solution of problem \eqref{problemalocale} if
\begin{equation}
\begin{aligned}
-\int_{\Omega}^{}\int_{0}^{\tau} \rho(x) u\, \varphi_t \,dt\,dx &= \int_{\Omega}^{}\int_{0}^{\tau}  u^m \Delta \varphi \,dt\,dx \\ &+ \int_{\Omega}^{}\int_{0}^{\tau} \rho(x) u^p \varphi \,dt\,dx
\end{aligned}
\end{equation}
for any $\varphi \in C_c^{\infty}(\overline{\Omega} \times [0,\tau))$ with $\varphi| _{\partial \Omega}=0$ for all $t\in [0,\tau)$. Moreover, we say that a nonnegative function $u\in L^{\infty}(\Omega\times (0,S))$ for any $S<\tau$ is a subsolution (supersolution) if it satisfies \eqref{veryweaklocale} with the inequality $"\le"$ ($"\ge"$) instead of $"="$, with $\varphi\geq 0$.
\label{soluzioneveryweaklocale}
\end{definition}

Finally, let us recall the following well-known criterion, that will be used in the sequel; we reproduce it for reader's convenience.
Let $\Omega\subseteq \mathbb R^N$ be an open set. Suppose that $\Omega=\Omega_1\cup \Omega_2$ with  $\Omega_1\cap \Omega_2=\emptyset$, and that   $\Sigma:=\partial \Omega_1\cap\partial \Omega_2$ is of class $C^1$. Let $n$ be the unit outwards normal to $\Omega_1$ at $\Sigma$.
Let
\begin{equation}\label{eq188}
u=\begin{cases}
u_1 & \textrm{in }\, \Omega_1\times [0, T),\\
u_2 & \textrm{in }\, \Omega_2\times [0, T)\,,
\end{cases}
\end{equation}
where $\partial_t u\in C(\Omega_1\times (0, T)), u_1^m\in C^2(\Omega_1\times (0, T))\cap C^1(\overline{\Omega}_1\times (0, T)) , \partial_t u_2\in C(\Omega_2\times (0, T)), )u_2^m\in C^2(\Omega_2\times (0, T))\cap C^1(\overline{\Omega}_2\times (0, T)).$

\begin{lemma}\label{lemext}
Let assumption \eqref{hp} be satisfied.

(i) Suppose that
\begin{equation}\label{eq185}
\begin{aligned}
&\partial_t u_1 \geq \frac 1{\rho}\Delta u_1^m +u_1^p \quad \textrm{for any}\,\,\, (x,t)\in \Omega_1\times (0, T),\\
&\partial_t u_2 \geq  \frac 1{\rho}\Delta u_2^m + u_2^p \quad \textrm{for any}\,\,\, (x,t)\in \Omega_2\times (0, T),
\end{aligned}
\end{equation}
\begin{equation}\label{eq186}
u_1=u_2, \quad \frac{\partial u_1^m}{\partial n}\geq \frac{\partial u_2^m}{\partial n}\quad \textrm{for any }\,\, (x,t)\in \Sigma\times (0, T)\,.
\end{equation}
Then $u$, defined in \eqref{eq188}, is a supersolution to equation \eqref{eq189}, in the sense of Definition \ref{soldom}.

(ii)  Suppose that
\begin{equation}\label{eq185b}
\begin{aligned}
&\partial_t u_1 \leq  \frac 1{\rho}\Delta u_1^m +u_1^p \quad \textrm{for any}\,\,\, (x,t)\in \Omega_1\times (0, T),\\
&\partial_t u_2 \leq  \frac 1{\rho}\Delta u_2^m + u_2^p \quad \textrm{for any}\,\,\, (x,t)\in \Omega_2\times (0, T),
\end{aligned}
\end{equation}
\begin{equation}\label{eq186b}
u_1=u_2, \quad\frac{\partial u_1^m}{\partial n}\leq \frac{\partial u_2^m}{\partial n}\quad \textrm{for any}\,\, (x,t)\in \Sigma\times (0, T)\,.
\end{equation}
Then $u$, defined in \eqref{eq188}, is a subsolution to equation \eqref{eq189}, in the sense of Defi\-nition \ref{soldom}.
\end{lemma}

\begin{proof}
Take any $\varphi \in C_c^{\infty}(\overline{\Omega} \times [0,\tau))$ with $\varphi| _{\partial \Omega}=0$ for all $t\in [0,\tau), \varphi\geq 0$.

(i) We multiply by $\varphi$ both sides of the two inequalities in \eqref{eq185}, then integrating two times by parts  we get
\begin{equation*}
\begin{aligned}
&-\int_0^{\tau}\int_{\Omega_1} \rho (u_1 \varphi_t + u_1^p \varphi) dx dt \\ &
\geq \int_0^{\tau} u_1^m \Delta \varphi dx dt -\int_0^{\tau} \int_{\Sigma} u_1^m\frac{\partial \varphi}{\partial n} d\sigma dt +\int_0^{\tau}\int_{\Sigma}\varphi \frac{\partial u_1^m}{\partial n} d\sigma dt\,,
\end{aligned}
\end{equation*}
\begin{equation*}
\begin{aligned}
&-\int_0^{\tau}\int_{\Omega_2} \rho (u_2 \varphi_t - u_2^p \varphi) dx dt \\ &
\geq \int_0^{\tau} u_2^m \Delta \varphi dx dt +\int_0^{\tau} \int_{\Sigma} u_2^m\frac{\partial \varphi}{\partial n} d\sigma dt -\int_0^{\tau}\int_{\Sigma}\varphi \frac{\partial u_2^m}{\partial n} d\sigma dt\,.
\end{aligned}
\end{equation*}
Summing up the previous two inequalities and using \eqref{eq186} we obtain
\[
-\int_0^{\tau}\int_{\Omega} \rho (u \varphi_t + u^p \varphi)\, dx dt \geq \int_0^{\tau} u^m \Delta \varphi\, dx dt\,.
\]
Hence the conclusion follows in this case.  The statement (ii) can be obtained in the same way. This completes the proof.
\end{proof}

\section{Global existence: proofs}\label{gepr}

\medskip

In what follows we set $r\equiv |x|$. We want to construct a suitable family of supersolutions of equation
\begin{equation}
u_t =\frac{1}{\rho(x)}\Delta(u^m)+u^p \quad \text{ in } \R^N\times(0,+\infty).
\label{equazione}
\end{equation}
To this purpose, we define, for all $(x,t) \in \big[\R^N \setminus B_1(0)\big] \times [0,+\infty)$,
\begin{equation}
{u}(x,t)\equiv u(r(x),t):=C\zeta(t)\left [1-\frac{r^b}{a}\eta(t)\right]_{+}^{\frac{1}{m-1}}\,,
\label{subsuper}
\end{equation}
where $\eta$, $\zeta \in C^1([0, +\infty); [0, +\infty))$ and $C > 0$, $a > 0$.

\smallskip

Now, we compute
$$
u_t - \frac{1}{\rho}\Delta(u^m)-u^p.
$$
To do this, let us set
$$
F(r,t):= 1-\frac{r^b}{a}\eta(t)
$$
and define
$$
D_1:=\left \{ (x,t) \in [\R^N \setminus B_1(0)] \times (0,+\infty)\, |\,\, 0<F(r,t)<1 \right \}.
$$
For any $(x,t) \in D_1$, we have:
\begin{equation}
\begin{aligned}
u_t &=C\zeta ' F^{\frac{1}{m-1}} + C\zeta \frac{1}{m-1} F^{\frac{1}{m-1}-1} \left ( -\frac{r^b}{a} \eta ' \right ) \\
&=C\zeta ' F^{\frac{1}{m-1}} + C\zeta \frac{1}{m-1} \left (1-\frac{r^b}{a} \eta \right ) \frac{\eta'}{\eta}F^{\frac{1}{m-1}-1} - C\zeta \frac{1}{m-1} \frac{\eta'}{\eta} F^{\frac{1}{m-1}-1} \\
&=C\zeta ' F^{\frac{1}{m-1}} + C\zeta \frac{1}{m-1} \frac{\eta'}{\eta}F^{\frac{1}{m-1}} - C\zeta \frac{1}{m-1} \frac{\eta'}{\eta} F^{\frac{1}{m-1}-1}; \\
\end{aligned}
\label{dertempo}
\end{equation}
\begin{equation}
(u^m)_r=-C^m \zeta^m \frac{m}{m-1} F^{\frac{1}{m-1}} \frac{b}{a}\eta r^{b-1};
\label{derprima}
\end{equation}
\begin{equation}
\begin{aligned}
(u^m)_{rr}&=-C^m \zeta^m \frac{m}{(m-1)^2} F^{\frac{1}{m-1}-1} \frac{b^2}{a}\eta r^{b-2} \left(1-\frac{r^b}{a} \eta\right ) \\&+ C^m \zeta^m \frac{m}{(m-1)^2} F^{\frac{1}{m-1}-1} \frac{b^2}{a}\eta r^{b-2}\\&- C^m \zeta^m \frac{m}{m-1} F^{\frac{1}{m-1}} \frac{b(b-1)}{a}\eta r^{b-2}.
\end{aligned}
\label{derseconda}
\end{equation}
\begin{equation}\label{laplaciano}
\begin{aligned}
\Delta({u}^m)
&=({u}^m)_{rr} + \frac{(N-1)}{r}({u}^m)_r  \\
&=C^m \zeta^m \frac{m}{(m-1)^2} F^{\frac{1}{m-1}-1} \frac{b^2}{a}\eta r^{b-2}
\\&- C^m \zeta^m \frac{m}{(m-1)^2} F^{\frac{1}{m-1}} \frac{b^2}{a}\eta r^{b-2} \\ &- C^m \zeta^m \frac{m}{m-1} F^{\frac{1}{m-1}} \frac{b(b-1)}{a}\eta r^{b-2} \\&+  \frac{(N-1)}{r}\left (-C^m \zeta^m \frac{m}{m-1} F^{\frac{1}{m-1}} \frac{b}{a}\eta r^{b-1} \right )  \\
&= C^m \zeta^m \frac{m}{(m-1)^2} \frac{b^2}{a}\eta F^{\frac{1}{m-1}-1}  r^{b-2} \\
&- C^m (N-2)\zeta^m \frac{m}{m-1} \frac{b}{a}\eta F^{\frac{1}{m-1}} r^{b-2} \\&- C^m \zeta^m \frac{m^2}{(m-1)^2} \frac{b^2}{a}\eta F^{\frac{1}{m-1}} r^{b-2}\,.
\end{aligned}
\end{equation}
We set $\overline u\equiv u$,
\begin{equation}
\overline w(x,t)\equiv \overline w(r(x),t) :=
\begin{cases}
\overline u(x,t) \quad \text{in } [\R^N \setminus B_{1}(0)] \times [0,+\infty), \\
\overline v(x,t) \quad \text{in } B_{1}(0) \times [0,+\infty),
\end{cases}
\label{w}
\end{equation}
where
\begin{equation}
\overline v(x,t) \equiv \overline v(r(x),t):= C\zeta(t) \left [ 1-\frac{(br^2+2-b)}{2}\frac{\eta(t)}{a} \right ]^{\frac{1}{m-1}}_{+}\,.
\label{v}
\end{equation}
We also define
\begin{equation}\label{coeff}
\begin{aligned}
& K:=  \left (\frac{m-1}{p+m-2}\right)^{\frac{m-1}{p-1}} - \left (\frac{m-1}{p+m-2}\right)^{\frac{p+m-2}{p-1}}>0 \,,\\
&\bar{\sigma}(t) := \zeta ' + \zeta \frac{1}{m-1} \frac{\eta'}{\eta} + C^{m-1} \zeta^m \frac{m}{m-1} \frac{b}{a}\eta k_1 \left (N-2+ \frac{bm}{m-1} \right ),\\
&\bar{\delta}(t) := \zeta \frac{1}{m-1} \frac{\eta'}{\eta} + C^{m-1} \zeta^m \frac{m}{(m-1)^2} \frac{b^2}{a}\eta k_2, \\
& \bar{\gamma}(t):=C^{p-1}\zeta^p(t)\,,\\
& \bar{\sigma}_0(t) := \zeta ' + \zeta \frac{1}{m-1} \frac{\eta'}{\eta} + C^{m-1} \zeta^m \frac{m}{m-1} N b \,k_1\frac{\eta}{a}, \\
& \bar{\delta}_0(t) := \zeta \frac{1}{m-1} \frac{\eta'}{\eta}+ C^{m-1}b^2\,k_2 \zeta^m \frac m{(m-1)^2}\frac{\eta^2}{a^2}.
\end{aligned}
\end{equation}

\begin{proposition}
Let $\zeta=\zeta(t)$, $\eta=\eta(t) \in C^1([0,+\infty);[0, +\infty))$.  Let $K, \bar\sigma, \bar\delta, \bar\gamma, \bar\sigma_0, \bar\delta_0$ be defined in \eqref{coeff}.
Assume \eqref{eq2}, \eqref{eq30b}, and that, for all $t\in (0,+\infty)$,
\begin{equation}\label{eq33}
\eta(t)<a\,,
\end{equation}
\begin{equation}
-\frac{\eta'}{\eta^2} \ge \frac{b^2}{a}C^{m-1} \zeta^{m-1}(t)\frac{m}{m-1}k_2,
\label{cond1super}
\end{equation}
\begin{equation}
\zeta' + C^{m-1} \zeta^m \frac{b}{a}\frac{m}{m-1} \eta \left[k_1 \left(N-2+\frac{bm}{m-1}\right)-\frac{k_2b}{m-1}\right] -C^{p-1}\zeta^p \ge 0,
\label{cond2super}
\end{equation}
\begin{equation}\label{eq30}
-\frac{\eta'}{\eta^3} \ge \frac{C^{m-1}}{a^2}\,k_2 \zeta^{m-1}\frac{m}{m-1},
\end{equation}
\begin{equation}\label{eq31}
\zeta' + N\zeta^m \frac{C^{m-1}}{a}\frac{m}{m-1} \eta\,k_1 -N \zeta^m \frac{C^{m-1} }{a^2}\frac{m}{(m-1)^2} \eta^2\,k_2 -C^{p-1}\zeta^p \ge 0.
\end{equation}
Then $w$ defined in \eqref{w} is a supersolution of equation \eqref{equazione}.
\label{propsupersolution}
\end{proposition}

\begin{proof}[Proof of Proposition \ref{propsupersolution}]
In view of \eqref{dertempo}, \eqref{derprima}, \eqref{derseconda} and \eqref{laplaciano}, for any $(x,t)\in D_1$,
\begin{equation}\label{eq21}
\begin{aligned}
&\bar u_t - \frac{1}{\rho}\Delta(\bar u^m)-\bar u^p\\
=
&\,\,C\zeta ' F^{\frac{1}{m-1}} + C\zeta \frac{1}{m-1} \frac{\eta'}{\eta}F^{\frac{1}{m-1}} - C\zeta \frac{1}{m-1} \frac{\eta'}{\eta} F^{\frac{1}{m-1}-1}\\ &- \frac{r^{b-2}}{\rho} \left \{ C^m \zeta^m \frac{m}{(m-1)^2} \frac{b^2}{a}\eta F^{\frac{1}{m-1}-1} - C^m (N-2) \zeta^m \frac{m}{m-1} \frac{b}{a}\eta F^{\frac{1}{m-1}} \right. \\ & \left . - C^m  \zeta^m \frac{m^2}{(m-1)^2} \frac{b^2}{a}\eta F^{\frac{1}{m-1}} \right \} - C^p \zeta^p F^{\frac{p}{m-1}}\,.
\end{aligned}
\end{equation}
Thanks to hypothesis \eqref{hp}, we have
\begin{equation}\label{eq20}
\frac{r^{b-2}}{\rho} \ge k_1, \quad
-\frac{r^{b-2}}{\rho} \ge -k_2\quad \textrm{for all}\,\,\, x\in \mathbb R^N\setminus B_1(0)\,.
\end{equation}
From \eqref{eq21} and \eqref{eq20} we get
\begin{equation}\label{eq22}
\begin{aligned}
&\bar u_t - \frac{1}{\rho}\Delta(\bar u^m)-\bar u^p\\
& \ge CF^{\frac{1}{m-1}-1} \left \{F\left [\zeta ' + \zeta \frac{1}{m-1} \frac{\eta'}{\eta} + C^{m-1} \zeta^m \frac{m}{m-1} \frac{b}{a}\eta k_1 \left (N-2+ \frac{bm}{m-1} \right ) \right ] \right . \\
& \left .-\zeta \frac{1}{m-1} \frac{\eta'}{\eta} - C^{m-1} \zeta^m \frac{m}{(m-1)^2} \frac{b^2}{a}\eta k_2 - C^{p-1} \zeta^p F^{\frac{p+m-2}{m-1}} \right \}\,.
\end{aligned}
\end{equation}
From \eqref{eq22}, taking advantage from $\bar{\sigma}(t)$, $\bar{\delta}(t)$ and $\bar{\gamma}(t)$ defined in \eqref{coeff}, we have
\begin{equation}\label{eq23}
\bar u_t - \frac{1}{\rho}\Delta(\bar u^m)-\bar u^p\geq C F^{\frac{1}{m-1}-1} \left[\bar{\sigma}(t)F - \bar{\delta}(t) - \bar{\gamma}(t)F^{\frac{p+m-2}{m-1}}\right].
\end{equation}
For each $t>0$, set
$$\varphi(F):=\bar{\sigma}(t)F - \bar{\delta}(t) - \bar{\gamma}(t)F^{\frac{p+m-2}{m-1}}, \quad F\in (0,1)\,.$$
Now our goal is to find suitable $C,a,\zeta, \eta$ such that, for each $t>0$,
$$\varphi(F) \ge 0 \quad \textrm{for any}\,\, F \in (0,1)\,.$$
We observe that $\varphi(F)$ is concave in the variable $F$, hence it is sufficient to have $\varphi(F)$ positive in the extrema of the interval of definition $(0,1)$. This reduces to the system
\begin{equation}
\begin{cases}
\varphi(0) \ge 0  \\
\varphi(1) \ge 0\,,
\end{cases}
\end{equation}
for each $t>0$. The system is equivalent to
$$
\begin{cases}
-\bar{\delta}(t) \ge 0  \\
\bar{\sigma}(t)-\bar{\delta}(t)-\bar{\gamma}(t) \ge 0\,,
\end{cases}
$$
that is
$$
\begin{cases}
-\frac{\eta'}{\eta^2} \ge \frac{b^2}{a} C^{m-1} \zeta^{m-1} \frac{m}{m-1} k_2\\
\quad \\
\zeta' + C^{m-1} \zeta^m \frac{b}{a}\frac{m}{m-1} \eta \left[k_1 \left(N-2+\frac{bm}{m-1}\right)-\frac{k_2b}{m-1}\right] -C^{p-1}\zeta^p \ge 0,
\end{cases}
$$
which is guaranteed by \eqref{eq2}, \eqref{cond1super} and \eqref{cond2super}.
Hence we have proved that
$$
\bar{u}_t - \frac{1}{\rho}\Delta(\bar{u}^m)-\bar{u}^p \ge 0 \quad \text{in }\,\,\,  D_1\,.
$$
Since $\bar u^m\in C^1([\R^N\setminus B_1(0)]\times (0, T))$, in view of Lemma \ref{lemext}-(i) (applied with $\Omega_1=D_1, \Omega_2=\R^N\setminus[B_1(0)\cup D_1], u_1=\bar u, u_2=0, u=\bar u)$, we can deduce that $\bar u$ is a supersolution of equation
\begin{equation}\label{eq29}
\bar{u}_t - \frac{1}{\rho}\Delta(\bar{u}^m)-\bar{u}^p = 0 \quad \text{in } \,\,\, \big[\mathbb R^N\setminus B_1(0)\big]\times (0, +\infty)\,,
\end{equation}
in the sense of Definition \ref{soldom}.  Now let $v$ be as in \eqref{v}. Set
$$
G(r,t):=1- \frac{b r^2+2-b}{2}\frac{\eta(t)}{a}.
$$
Due to \eqref{eq33},
\[0<G(r,t)<1\quad \textrm{for all}\,\,\, (x,t)\in B_1(0)\times (0,+\infty)\,.\]
For any $(x,t) \in B_1(0)\times (0, +\infty)$, we have:
\begin{equation}
\bar v_t =C\zeta ' G^{\frac{1}{m-1}} + C\zeta \frac{1}{m-1} \frac{\eta'}{\eta}G^{\frac{1}{m-1}} - C\zeta \frac{1}{m-1} \frac{\eta'}{\eta} G^{\frac{1}{m-1}-1};
\label{dertempoG}
\end{equation}
\begin{equation}
(\bar v^m)_r=-C^m b \zeta^m \frac{m}{m-1} G^{\frac{1}{m-1}} \frac{\eta}{a} r;
\label{derprimaG}
\end{equation}
\begin{equation}
\begin{aligned}
(\bar v^m)_{rr}&= C^m \zeta^m \frac{m}{(m-1)^2} G^{\frac{1}{m-1}-1} \frac{\eta^2}{a^2} b^2 r^2 - C^m b \zeta^m \frac{m}{m-1} G^{\frac{1}{m-1}} \frac{\eta}{a}.
\end{aligned}
\label{dersecondaG}
\end{equation}
Therefore, for all $(x,t)\in B_1(0)\times (0, +\infty)$,
\begin{equation}\label{eq22bis}
\begin{aligned}
&\bar v_t- \frac{1}{\rho}\Delta(\bar v^m) - \bar v^p \\
&=C G^{\frac{1}{m-1}-1}\Big\{G\left[\zeta' + \frac{\zeta}{m-1}\frac{\eta'}{\eta}+ b \frac{N-1}{r}C^{m-1}\zeta^m \frac{m}{m-1}\frac{r}{\rho}\frac{\eta}{a} +\frac b{\rho}C^{m-1}\zeta^m\frac{m}{m-1}\frac{\eta}{a}\right]\\
&-\frac{\zeta}{m-1}\frac{\eta'}{\eta} -\frac{r^2}{\rho}b^2C^{m-1}\frac{m}{(m-1)^2}\zeta^m \frac{\eta^2}{a^2}-C^{p-1}\zeta^p G^{\frac{p+m-2}{m-1}} \Big\}\,.
\end{aligned}
\end{equation}
Using \eqref{hyprho} and the fact that $r\in (0,1)$, \eqref{eq22bis} yields, for all $(x,t)\in B_1(0)\times (0, +\infty)$,
\begin{equation}\label{eq34}
\begin{aligned}
\bar v_t- \frac{1}{\rho}&\Delta(\bar v^m) - \bar v^p  \\
&\ge C G^{\frac{1}{m-1}-1}\Big\{G\left[\zeta' + \frac{\zeta}{m-1}\frac{\eta'}{\eta}+ N b \,k_1C^{m-1}\zeta^m \frac{m}{m-1}\frac{\eta}{a} \right]\\
&-\frac{\zeta}{m-1}\frac{\eta'}{\eta} -C^{m-1}b^2 \,k_2\frac{m}{(m-1)^2}\frac{\eta^2}{a^2}-C^{p-1}\zeta^p G^{\frac{p+m-2}{m-1}} \Big\}\\
&=C G^{\frac{1}{m-1}-1}\left[\bar{\sigma}_0(t) G - \bar{\delta}_0(t) - \bar{\gamma}(t)G^{\frac{p+m-2}{m-1}} \right]\,.
\end{aligned}
\end{equation}
Hence, due to \eqref{eq34}, we obtain for all $(x,t)\in B_1(0)\times (0, +\infty)$,
\begin{equation}\label{eq23}
\bar v_t - \frac{1}{\rho}\Delta( \bar v^m)-\bar v^p\geq C G^{\frac{1}{m-1}-1} \left[\bar{\sigma_0}(t)G - \bar{\delta_0}(t) - \bar{\gamma}(t)G^{\frac{p+m-2}{m-1}}\right].
\end{equation}
For each $t>0$, set
$$\psi(G):=\bar{\sigma}_0(t)G - \bar{\delta}_0(t) - \bar{\gamma}(t)G^{\frac{p+m-2}{m-1}}, \quad G\in (0,1)\,.$$
Now our goal is to verify that, for each $t>0$,
$$\psi(G) \ge 0 \quad \textrm{for any}\,\, G \in (0,1)\,.$$
We observe that $\psi(G)$ is concave in the variable $G$, hence it is sufficient to have $\psi(G)$ positive in the extrema of the interval of definition $(0,1)$. This reduces to the system
\begin{equation}
\begin{cases}
\psi(0) \ge 0 \\
\psi(1) \ge 0 \,,
\end{cases}
\end{equation}
for each $t>0$. The system is equivalent to
$$
\begin{cases}
-\bar{\delta_0}(t) \ge 0  \\
\bar{\sigma_0}(t)-\bar{\delta_0}(t)-\bar{\gamma}(t) \ge 0\,,
\end{cases}
$$
that is
$$
\begin{cases}
-\frac{\eta'}{\eta^3} \ge b^2 \frac{C^{m-1}}{a^2} \,k_2 \zeta^{m-1} \frac{m}{m-1} \\
\quad \\
\zeta' + \frac{C^{m-1}}{a} b N\,k_1\zeta^m \frac{m}{m-1} \eta - b^2\frac{C^{m-1}}{a^2} \,k_2 \zeta^m \frac{m}{(m-1)^2} \eta^2 -C^{p-1}\zeta^p \ge 0,
\end{cases}
$$
which is guaranteed by \eqref{eq2}, \eqref{eq30} and \eqref{eq31}.
Hence we have proved that
\begin{equation}\label{eq35}
\bar v_t- \frac{1}{\rho}\Delta(\bar v^m) - \bar v^p \ge 0 \quad \text{for all}\,\,\, (x,t)\in B_1(0)\times (0, +\infty)
\end{equation}
Now, observe that $\bar w \in C(\R^N \times [0,+\infty))$; indeed,
$$
\bar{u} = \bar v  = C \zeta(t) \left [ 1- \frac{\eta(t)}{a} \right ]_+^{\frac{1}{m-1}} \quad \textrm{in}\,\, \partial B_1(0)\times (0, +\infty)\,.
$$
Moreover, $\bar w^m \in C^1(\R^N \times [0,+\infty))$; indeed,
\begin{equation}\label{eq180}
(\bar{u}^m)_r  = (\bar v^m)_r  = -C^m \zeta(t)^m \frac{m}{m-1}b  \frac{\eta(t)}{a} \left [ 1- \frac{\eta(t)}{a} \right ]_+^{\frac{1}{m-1}} \,\, \textrm{in}\,\, \partial B_1(0)\times (0, +\infty)\,.
\end{equation}
In conclusion,  by \eqref{eq29}, \eqref{eq34}, \eqref{eq180} and Lemma \ref{lemext}-(i) (applied with $\Omega_1=\R^N\setminus B_1(0), \Omega_2=B_1(0), u_1=\bar u, u_2=\bar v, u= \bar w)$, we can infer  that $\bar w$ is a supersolution to equation \eqref{equazione} in the sense of Definition \ref{soldom}.
\end{proof}

\begin{remark}\label{ParThmGlob}
Let
\[p>\overline p,\]
and assumptions \eqref{eq2} and \eqref{eq30b} be satisfied.
Let $\omega:=\frac{C^{m-1}}{a}$. In Theorem \ref{teosupersolution}, the precise hypotheses on parameters $\alpha, \beta, C>0, \omega>0, T>0$  are the following:
\[\textrm{condition}\,\,  \eqref{hpalpha},
\]
\begin{equation}\label{eq3}
\beta - b^2 \omega k_2\frac{m}{m-1} \ge 0\,,
\end{equation}
\begin{equation}\label{eq4}
-\alpha + b\omega\frac{m}{m-1} \left [k_1 \left (N-2+\frac{bm}{m-1} \right ) - \frac{k_2b}{m-1} \right ] \ge C^{p-1}\,,
\end{equation}
\begin{equation}\label{eq54}
\beta T^{\beta} \,\,\, \ge \,\,\, b^2 \frac{\omega}{a}\,k_2\frac{m}{m-1},
\end{equation}
\begin{equation}\label{eq54b}
T^\beta > \frac {r_0}a\,\quad (\textrm{for }\, r_0>1),
\end{equation}
\begin{equation}\label{eq55b}
-\alpha+b\omega\frac{m}{m-1}\left(k_1N-b\frac{T^{-\beta}}{(m-1)a}\,k_2\right)\, \ge \,C^{p-1}\,.
\end{equation}
\end{remark}

\begin{lemma}\label{lemt1}
All the conditions in Remark \ref{ParThmGlob} can be satisfied simoultaneously.
\end{lemma}
\begin{proof}
We take $\alpha$ satisfying \eqref{hpalpha} and
\begin{equation}\label{sistema12}
\alpha <\min\left\{ \frac{k_1\left(N-2+\frac{bm}{m-1}\right) - \frac{k_2b}{m-1}}{k_1\left[m\left(N-2+b\right)-\left(N-2\right)\right ]}, \frac{k_1 N}{bk_2 + (m-1)k_1 N}, \frac 1{m-1}\right\}\,.
\end{equation}
This is possible, since
$$p>\overline p>m+\frac{k_2 b}{k_1 N}>m\,.$$
In view of \eqref{sistema12}, \eqref{eq2} and the fact that $\beta=1-\alpha(m-1)$, we can take $\omega>0$ so that \eqref{eq3} holds, the left-hand-side of \eqref{eq4} is positive, and
\begin{equation*}\label{eq220}
-\alpha+b\omega\frac{m}{m-1}(k_1 N -\epsilon)>0\,,
\end{equation*}
for some $\epsilon>0$. Then, we choose $C>0$ so small that \eqref{eq4} holds and
\begin{equation}\label{eq221}
-\alpha+b\omega\frac{m}{m-1}(k_1 N -\epsilon)> C^{p-1};
\end{equation}
therefore, also $a>0$ is properly fixed, in view of the definition of $\omega$.
We select $T>0$ so big that  \eqref{eq54}, \eqref{eq54b} are valid and
\begin{equation}\label{eq222}
k_1N-b\frac{T^{-\beta}}{(m-1)a}\,k_2\geq \epsilon\,.
\end{equation}
From \eqref{eq222} and \eqref{eq221} inequality \eqref{eq55b} follows.
\end{proof}

\begin{proof}[Proof of Theorem \ref{teosupersolution}] We prove Theorem \ref{teosupersolution} by means of Proposition \ref{teosupersolution}.
In view of Lemma \ref{lemt1}, we can assume that all the conditions of Remark \ref{ParThmGlob} are fulfilled.

Set $$\zeta(t)=(T+t)^{-\alpha}, \quad \eta(t)=(T+t)^{-\beta}, \quad \textrm{for all}\,\,\, t>0\,.$$
Observe that condition \eqref{eq54b} implies \eqref{eq33}. Moreover, consider conditions
\eqref{cond1super}, \eqref{cond2super} of Proposition \ref{propsupersolution} with this choice of $\zeta(t)$ and $\eta(t)$. Therefore we obtain
\begin{equation}\label{eq51}
\beta -\frac{b^2}{a} C^{m-1}\frac{m}{m-1}k_2 (T+t)^{-\alpha(m-1)-\beta+1} \ge 0
\end{equation}
and
\begin{equation}
\begin{aligned}
&-\alpha (T+t)^{-\alpha-1}+\frac{C^{m-1}}{a}\frac{mb}{m-1}\left [k_1 \left (N-2+\frac{bm}{m-1} \right ) - \frac{k_2b}{m-1} \right ] (T+t)^{-\alpha m-\beta}\\& -C^{p-1}(T+t)^{-\alpha p} \ge 0\,.
\end{aligned}
\label{eq50}
\end{equation}
Since, $\beta=1-\alpha(m-1)$, \eqref{eq51} and \eqref{eq50} become
\begin{equation}\label{eq52}
C^{m-1}\frac{m}{m-1}\frac{b}{a} \,\,\, \le \,\,\, \frac{1-\alpha(m-1)}{k_2b}\,,
\end{equation}
\begin{equation}
\begin{aligned}
&\left\{-\alpha +b\frac{C^{m-1}}{a}\frac{m}{m-1}\left [k_1 \left (N-2+\frac{bm}{m-1} \right ) - \frac{k_2b}{m-1} \right ] \right \}(T+t)^{-\alpha-1}\\ & \ge C^{p-1}(T+t)^{-\alpha p}\,.
\label{eq53}
\end{aligned}
\end{equation}
Due to assumption \eqref{hpalpha},
\begin{equation}\label{eq201}
\beta >0, \quad -\alpha-1 \ge - p\alpha.
\end{equation}
Thus \eqref{eq52} and \eqref{eq53} follow from \eqref{eq201}, \eqref{eq3} and \eqref{eq4}.

%

\medskip

We now consider conditions \eqref{eq30} and \eqref{eq31} of Proposition \ref{propsupersolution}.
Substituting $\zeta(t)$, $\eta(t)$, $\alpha$ and $\beta$ previously chosen, we get \eqref{eq54} and
\begin{equation}\label{eq55}
\left[-\alpha+b\frac{C^{m-1}}{a}\frac{m}{m-1}\left(k_1N-b\frac{(T+t)^{-\beta}}{(m-1)a}\,k_2\right)\right](T+t)^{-\alpha-1}\, \ge \,C^{p-1}(T+t)^{-p\alpha}\,.
\end{equation}
Condition \eqref{eq55} follows from \eqref{eq201} and \eqref{eq55b}.

Hence, we can choose $\alpha, \beta$, $C>0$, $a>0$ and $T$ so that
\eqref{eq52}, \eqref{eq53}, \eqref{eq54} and \eqref{eq55} hold.  Thus the conclusion follows by Propositions \ref{propsupersolution} and \ref{cpsup}.
\end{proof}

\section{Blow-up: proofs}\label{bupr}
Let
\begin{equation}
\underline w(x,t)\equiv \underline  w(r(x),t) :=
\begin{cases}
\underline  u(x,t) \quad \text{in } [\R^N \setminus B_{1}(0)] \times [0, T), \\
\underline  v(x,t) \quad \text{in } B_{1}(0) \times [0,T),
\end{cases}
\label{wb}
\end{equation}
where  $\underline u\equiv u$ is defined in \eqref{subsuper}
and
$\underline  v$ is defined as follows
\begin{equation}
\underline  v(x,t) \equiv \underline  v(r(x),t):= C\zeta(t) \left [ 1-r^2\frac{\eta(t)}{a} \right ]^{\frac{1}{m-1}}_{+}\,.
\label{vb}
\end{equation}
Observe that  for any $(x,t) \in B_1(0)\times (0, T)$, we have:
\begin{equation}
\underline v_t =C\zeta ' G^{\frac{1}{m-1}} + C\zeta \frac{1}{m-1} \frac{\eta'}{\eta}G^{\frac{1}{m-1}} - C\zeta \frac{1}{m-1} \frac{\eta'}{\eta} G^{\frac{1}{m-1}-1};
\label{eq90}
\end{equation}
\begin{equation*}
(\underline  v^m)_r=-2 C^m  \zeta^m \frac{m}{m-1} G^{\frac{1}{m-1}} \frac{\eta}{a} r;
\end{equation*}
\begin{equation*}
\begin{aligned}
(\underline  v^m)_{rr}&= 4 C^m \zeta^m \frac{m}{(m-1)^2} G^{\frac{1}{m-1}-1} \frac{\eta}{a}   - 2 C^m  \zeta^m \frac{m}{m-1} G^{\frac{1}{m-1}} \frac{\eta}{a} \\&- 4 C^m \zeta^m\frac{m}{(m-1)^2}\frac{\eta}{a}G^{\frac 1{m-1}},
\end{aligned}
\end{equation*}
\begin{equation}\label{eq91}
\begin{aligned}
\Delta(\underline  v^m)&=4 C^m \zeta^m \frac{m}{(m-1)^2} G^{\frac{1}{m-1}-1} \frac{\eta}{a} - 4 C^m \zeta^m\frac{m}{(m-1)^2}\frac{\eta}{a}G^{\frac 1{m-1}}\\
& - 2 N C^m  \zeta^m \frac{m}{m-1} G^{\frac{1}{m-1}} \frac{\eta}{a}\,.
\end{aligned}
\end{equation}

Therefore, from \eqref{eq90} and \eqref{eq91} we get, for all $(x,t)\in B_1(0)\times (0, T)$,
\begin{equation}\label{eq92}
\begin{aligned}
&\underline  v_t- \frac{1}{\rho}\Delta(\underline  v^m) - \underline  v^p \\
&=C G^{\frac{1}{m-1}-1}\Big\{G\left[\zeta' + \frac{\zeta}{m-1}\frac{\eta'}{\eta}+ 2 NC^{m-1}\zeta^m \frac{m}{m-1}\frac 1{\rho}\frac{\eta}{a} +\frac 4{\rho}C^{m-1}\zeta^m\frac{m}{(m-1)^2}\frac{\eta}{a}\right]\\
&-\frac{\zeta}{m-1}\frac{\eta'}{\eta} -\frac{4}{\rho}C^{m-1}\frac{m}{(m-1)^2}\frac{\eta}{a}-C^{p-1}\zeta^p G^{\frac{p+m-2}{m-1}} \Big\}\,.
\end{aligned}
\end{equation}

We also define
\begin{equation}
\begin{aligned}
& \underline{\sigma}(t) := \zeta ' + \zeta \frac{1}{m-1} \frac{\eta'}{\eta} + C^{m-1} \zeta^m \frac{m}{m-1} \frac{b}{a}\eta k_2 \left (N-2+ \frac{bm}{m-1} \right ),\\
& \underline{\delta}(t) := \zeta \frac{1}{m-1} \frac{\eta'}{\eta} + C^{m-1} \zeta^m \frac{m}{(m-1)^2} \frac{b^2}{a}\eta k_1, \\
& \underline{\gamma}(t):=C^{p-1} \zeta^p, \\
& \underline{\sigma}_0(t) := \zeta ' + \zeta \frac{1}{m-1} \frac{\eta'}{\eta} + 2C^{m-1} \zeta^m \frac{m}{m-1} \left( N+ \frac 2{m-1}\right)\rho_2\frac{\eta}{a}, \\
& \underline{\delta}_0(t) := \zeta \frac{1}{m-1} \frac{\eta'}{\eta}+ 4\frac{C^{m-1}}{a}\zeta^m \rho_1\frac{m}{(m-1)^2}\eta , \\
& \textit{K}:=  \left (\frac{m-1}{p+m-2}\right)^{\frac{m-1}{p-1}} - \left (\frac{m-1}{p+m-2}\right)^{\frac{p+m-2}{p-1}} >0.
\end{aligned}
\label{definitionssubsolution}
\end{equation}

\begin{proposition}
Let $T\in (0,\infty)$, $\zeta$, $\eta \in C^1([0,T);[0, +\infty))$.
Let $\underline{\sigma},\underline{\delta},\underline{\gamma},\underline{\sigma}_0,\underline{\delta}_0, \textit{K}$ be defined in \eqref{definitionssubsolution}. Assume \eqref{eq30b} and that, for all $t\in(0,T)$,
\begin{equation}\label{eq76}
\textit{K}[\underline{\sigma}(t)]^{\frac{p+m-2}{p-1}} \le \underline{\delta}(t) [\underline{\gamma}(t)]^{\frac{m-1}{p-1}}, \end{equation}
\begin{equation}\label{eq77}
(m-1) \underline{\sigma}(t) \le (p+m-2) \underline{\gamma}(t)\,,
\end{equation}
\begin{equation}\label{eq76b}
\textit{K}[\underline{\sigma_0}(t)]^{\frac{p+m-2}{p-1}} \le \underline{\delta_0}(t) [\underline{\gamma}(t)]^{\frac{m-1}{p-1}}, \end{equation}
\begin{equation}\label{eq77b}
(m-1) \underline{\sigma_0}(t) \le (p+m-2) \underline{\gamma}(t)\,.
\end{equation}

Then $w$ defined in \eqref{wb} is a weak subsolution of equation \eqref{equazione}.
\label{propsubsolution}
\end{proposition}

\begin{proof}[Proof of Proposition \ref{propsubsolution}]
In view of \eqref{dertempo}, \eqref{derprima}, \eqref{derseconda} and \eqref{laplaciano}
we obtain
\begin{equation}\label{eq71}
\begin{aligned}
&\underline u_t - \frac{1}{\rho}\Delta(\underline u^m)- \underline u^p\\
&=C\zeta ' F^{\frac{1}{m-1}} + C\zeta \frac{1}{m-1} \frac{\eta'}{\eta}F^{\frac{1}{m-1}} - C\zeta \frac{1}{m-1} \frac{\eta'}{\eta} F^{\frac{1}{m-1}-1}\\ &- \frac{r^{b-2}}{\rho} \left \{ C^m \zeta^m \frac{m}{(m-1)^2} \frac{b^2}{a}\eta F^{\frac{1}{m-1}-1} - C^m \zeta^m \frac{m}{m-1} \frac{b}{a}\eta F^{\frac{1}{m-1}} - C^m \zeta^m \frac{m^2}{(m-1)^2} \frac{b^2}{a}\eta F^{\frac{1}{m-1}} \right \} \\ &- C^p \zeta^p F^{\frac{p}{m-1}}\quad
\textrm{for all }\, (x,t)\in D_1\,.
\end{aligned}
\end{equation}
In view of hypothesis \eqref{hp}, we can infer that
\begin{equation}\label{eq70}
\frac{r^{b-2}}{\rho} \le k_2, \quad  -\frac{r^{b-2}}{\rho} \le -k_1\quad \textrm{for all}\,\,\, x\in \mathbb R^N\setminus B_1(0)\,.
\end{equation}
From \eqref{eq71} and \eqref{eq70} we have
\begin{equation}\label{eq72}
\begin{aligned}
&\underline u_t - \frac{1}{\rho}\Delta(\underline u^m)- \underline u^p\\
& \le CF^{\frac{1}{m-1}-1} \left \{F\left [\zeta ' + \zeta \frac{1}{m-1} \frac{\eta'}{\eta} + C^{m-1} \zeta^m \frac{m}{m-1} \frac{b}{a}\eta k_2 \left (N-2+ \frac{bm}{m-1} \right ) \right ] \right . \\
& \left .-\zeta \frac{1}{m-1} \frac{\eta'}{\eta} - C^{m-1} \zeta^m \frac{m}{(m-1)^2} \frac{b^2}{a}\eta k_1 - C^{p-1} \zeta^p F^{\frac{p+m-2}{m-1}} \right \}.
\end{aligned}
\end{equation}
Thanks to \eqref{definitionssubsolution}, \eqref{eq72} becomes
\begin{equation}
\begin{aligned}
\underline u_t - \frac{1}{\rho}\Delta(\underline u^m)- \underline u^p\leq
 CF^{\frac{1}{m-1}-1} \varphi(F),
\end{aligned}
\end{equation}
where, for each $t\in (0, T)$,
$$\varphi(F):=\underline{\sigma}(t)F - \underline{\delta}(t) - \underline{\gamma}(t)F^{\frac{p+m-2}{m-1}}.$$
Our goal is to find suitable $C,a,\zeta, \eta$ such that, for each $t\in (0, T)$,
$$\varphi(F) \le 0 \quad \textrm{for any}\,\,\, F \in (0,1)\,.$$
To this aim, we impose that
\[\sup_{F\in (0,1)}\varphi(F)=\max_{F\in (0,1)}\varphi(F)= \varphi (F_0)\leq 0\,,\]
for some $F_0\in (0,1).$
We have
$$ \begin{aligned} \frac{d \varphi}{dF}=0 &\iff \underline{\sigma}(t) - \frac{p+m-2}{m-1} \underline{\gamma}(t) F^{\frac{p-1}{m-1}} =0 \\ & \iff F=F_0= \left [\frac{m-1}{p+m-2} \frac{\underline{\sigma}(t)}{\underline{\gamma}(t)} \right ]^{\frac{m-1}{p-1}} \,.\end{aligned}$$
Then
$$ \varphi(F_0)= K\, \frac{\underline{\sigma}(t)^{\frac{p+m-2}{p-1}}}{\underline{\gamma}(t)^{\frac{m-1}{p-1}}} - \underline{\delta}(t)\,, $$
where the coefficient $K$ depending on $m$ and $p$ has been defined in \eqref{definitionssubsolution}. By hypoteses \eqref{eq76} and \eqref{eq77}, for each $t\in (0, T)$,
\begin{equation}\label{eq74}
\varphi(F_0) \le 0\,,\quad
F_0 \le 1\,.
\end{equation}
So far, we have proved that
\begin{equation}\label{eq81}
\underline{u}_t-\frac{1}{\rho(x)}\Delta(\underline{u}^m)-\underline{u}^p \le 0 \quad \text{ in }\,\, D_1.
\end{equation}
Furthermore, since $\underline u^m\in C^1([\R^N\setminus B_1(0)]\times (0, T))$, due to Lemma \ref{lemext} (applied with $\Omega_1=D_1, \Omega_2=\R^N\setminus[B_1(0)\cup D_1], u_1=\underline u, u_2=0, u=\underline u$), it follows that $\underline u$ is a subsolution to equation
\[\underline{u}_t-\frac{1}{\rho(x)}\Delta(\underline{u}^m)-\underline{u}^p = 0 \quad \text{ in }\,\, [\mathbb R^N\setminus B_1(0)]\times (0, T),\]
in the sense of Definition \ref{soldom}.

\normalcolor

%
%
Let
\[D_2:=\{(x,t)\in B_1(0)\times (0, T)\,:\, 0<G(r, t)<1\}\,.\]
Using \eqref{hyprho}, \eqref{eq92} yields, for all $(x,t)\in D_2$,
\begin{equation}\label{eq93}
\begin{aligned}
v_t- \frac{1}{\rho}&\Delta(v^m) - v^p  \\
&\le C G^{\frac{1}{m-1}-1}\Big\{G\left[\zeta' + \frac{\zeta}{m-1}\frac{\eta'}{\eta}+ 2\left( N +\frac 2{m-1}\right) \,k_2 C^{m-1}\zeta^m \frac{m}{m-1}\frac{\eta}{a} \right]\\
&-\frac{\zeta}{m-1}\frac{\eta'}{\eta} - 4C^{m-1} \,k_1\frac{m}{(m-1)^2}\frac{\eta}{a}-C^{p-1}\zeta^p G^{\frac{p+m-2}{m-1}} \Big\}\\
&=C G^{\frac{1}{m-1}-1}\left[\underline{\sigma}_0(t) G - \underline{\delta}_0(t) - \underline{\gamma}(t)G^{\frac{p+m-2}{m-1}} \right]\,.
\end{aligned}
\end{equation}
Now, by the same arguments used to obtain \eqref{eq81}, in view of \eqref{eq76b} and \eqref{eq77b} we can infer that
\begin{equation}\label{eq94}
\underline  v_t - \frac 1{\rho}\Delta \underline  v^m \leq \underline  v^p\quad \textrm{for any}\,\,\, (x,t)\in D_2\,.
\end{equation}
Moreover, since $\underline  v^m\in C^1(B_1(0)\times (0, T))$, in view of Lemma \ref{lemext} (applied with $\Omega_1=D_2, \Omega_2=B_1(0)\setminus D_2, u_1=\underline v, u_2=0, u=\underline  v$), we get that
$\underline v$ is a subsolution to equation
\begin{equation}\label{eq94b}
\underline v_t  - \frac 1{\rho}\Delta \underline v^m = \underline v^p\quad \textrm{in}\,\,\, B_1(0)\times (0, T)\,,
\end{equation}
in the sense of Definition \ref{soldom}. Now, observe that $\underline  w \in C(\R^N \times [0,T))$; indeed,
$$
\underline{u} = \underline  v  = C \zeta(t) \left [ 1- \frac{\eta(t)}{a} \right ]_+^{\frac{1}{m-1}} \quad \textrm{in}\,\, \partial B_1(0)\times (0, T)\,.
$$
Moreover, since $b\in (0,2]$,
\begin{equation}\label{eq181}
(\underline{u}^m)_r  \geq  (\underline  v^m)_r  = -2C^m \zeta(t)^m \frac{m}{m-1} \frac{\eta(t)}{a} \left [ 1- \frac{\eta(t)}{a} \right ]_+^{\frac{1}{m-1}} \quad \textrm{in}\,\, \partial B_1(0)\times (0, T)\,.
\end{equation}
In conclusion, in view of \eqref{eq181} and Lemma \ref{lemext} (applied with $\Omega_1=B_1(0), \Omega_2=\R^N\setminus B_1(0), u_1=\underline  v, u_2=\underline u, u=\underline  w$), we can infer that $\underline  w$ is a subsolution to equation \eqref{equazione}, in the sense of Definition \ref{soldom}.
\end{proof}

\begin{remark}\label{PartThmBlo} Let $\omega:=\frac{C^{m-1}}{a}.$ In Theorem \ref{teosubsolution}
the precise choice of the parameters $C>0, a>0, T>0$ are as follows.
\begin{itemize}
\item[(a)] Let $p>m$. We require that
\begin{equation}\label{eq7}
\begin{aligned}
K &\left \{ \frac{1}{m-1} +  b k_2\omega \frac{m}{m-1}\left(\frac{bm}{m-1}+N-2\right) \right\}^{\frac{p+m-2}{p-1}}\\& \le \frac{C^{m-1}}{m-1} \left [b^2 k_1 \omega\frac{m}{m-1}   + \frac{p-m}{p-1} \right ] \,,
\end{aligned}
\end{equation}
\begin{equation}\label{eq8}
1+ \omega m b k_2 \left (N-2+\frac{bm}{m-1} \right )  \le \left (p+m-2 \right )C^{p-1}\,,
\end{equation}
\begin{equation}\label{eq8bis}
\begin{aligned}
&K\left[ \frac1{m-1} + 2\,k_2 \omega  \frac{m}{m-1}\left( N + \frac 2{m-1}\right) \right]^{\frac{p+m-2}{p-1}}\\&\leq \frac{C^{m-1}}{m-1}\left[4\,k_1 \omega\frac m{m-1} + \frac{p-m}{p-1} \right] \,,
\end{aligned}
\end{equation}
\begin{equation}\label{eq8tris}
1+ k_2 \omega \left( N + \frac 2{m-1}\right) \leq (p+m-2) C^{p-1}\,;
\end{equation}

\item[(b)] Let $p<m$. We require that
\begin{equation}\label{eq148}\begin{aligned}
& \omega > \frac{(m-p)(m-1)}{b^2(p-1)m k_1}, \\
\end{aligned}
\end{equation}
\begin{equation}\label{eq9}
\begin{aligned}
a \ge \max&\left\{\frac{K \left \{ \frac{1}{m-1} + \omega k_2\frac{m}{m-1} b \left(N-2+\frac{bm}{m-1}\right) \right\}^{\frac{p+m-2}{p-1}}}
{ \omega \frac{1}{m-1} \left [\omega \frac{m}{m-1} k_1 b^2 - \frac{m-p}{p-1} \right ] },\right. \\& \left.  \frac{K \left \{ \frac{1}{m-1} +2 \omega \,k_2\frac{m}{m-1}  \left(N+\frac 2{m-1}\right) \right\}^{\frac{p+m-2}{p-1}}}
{ \omega \frac{1}{m-1} \left [4\,k_1 \omega \frac{m}{m-1}  - \frac{m-p}{p-1} \right ] } \right\}\,,
\end{aligned}
\end{equation}
\begin{equation}\label{eq10}
\begin{aligned}
\left (p+m-2 \right ) \left (a\omega \right)^{\frac{p-1}{m-1}} \ge &  \max\left\{ 1 + \omega m\, b \,k_2\left(\frac{bm}{m-1} +N-2 \right), \right. \\ &\left. 1 + \omega   \,k_2\left(N+\frac{2}{m-1} \right)\right\}\,.
\end{aligned}
\end{equation}

\item[(c)] Let $p=m$. We require that $\omega>0$,
\begin{equation}\label{eq9c}
\begin{aligned}
a \ge \max&\left\{\frac{K \left\{ \frac{1}{m-1} + \omega k_2\frac{m}{m-1} b \left(N-2+\frac{bm}{m-1}\right) \right\}^{2}}
  {b^2 k_1 \omega^2    \frac{m}{(m-1)^2}} \, ,\right. \\& \frac{K \left \{ \frac{1}{m-1} +2 \omega \,k_2\frac{m}{m-1}  \left(N+\frac 2{m-1}\right) \right\}^{2}}
{  4 \,k_1 \omega^2 \frac{m}{(m-1)^2} } \,, \\
&\frac{1}{2(m-1)\omega}\left[ 1 + \omega \,m\, b\, k_2\left(\frac{bm}{m-1} +N-2 \right)\right]\,, \\
&\left.\frac{1}{2(m-1)\omega}\left[1 + \omega   \,k_2\left(N+\frac{2}{m-1} \right)\right]\right\}\,.
\end{aligned}
\end{equation}
\end{itemize}
\end{remark}

\begin{lemma}\label{lemt2}
All the conditions of Remark \ref{PartThmBlo} can hold simultaneously.
\end{lemma}
\begin{proof}
(a) We take any $\omega>0$, then we select $C>0$ big enough (therefore, $a>0$ is also fixed, due to the definition of $\omega$) so that \eqref{eq7}-\eqref{eq8tris} hold.

\smallskip

\noindent (b) We can take $\omega>0$ so that \eqref{eq148} holds, then we take $a>0$ sufficiently large to guarantee \eqref{eq9} and \eqref{eq10} (therefore, $C>0$ is also fixed).

\smallskip

\noindent (c) For any $\omega>0$, we take $a>0$ sufficiently large to guarantee \eqref{eq9c} (thus, $C>0$ is also fixed).
\end{proof}

\begin{proof}[Proof of Theorem \ref{teosubsolution}] We now prove Theorem \ref{teosubsolution}, by means of Proposition \ref{propsubsolution}. In view of Lemma \ref{lemt2}, we can assume that all the conditions in Remark \ref{PartThmBlo} are fulfilled.
Set
$$\zeta(t)=(T-t)^{-\alpha}\,, \quad \eta(t)=(T-t)^{\beta}$$ and $$\alpha=\frac{1}{p-1}\,,\quad \beta=\frac{m-p}{p-1}\,.$$ Then
\begin{equation*}
\underline{\sigma}(t) = \left [\frac{1}{m-1} + C^{m-1} \frac{m}{m-1} \frac{b}{a} k_2 \left (N-2+ \frac{bm}{m-1} \right ) \right ] \left (T-t \right)^{\frac{-p}{p-1}}\,,
\end{equation*}
\begin{equation*}
\underline{\delta}(t) :=  \left [\frac{m-p}{(m-1)(p-1)} + C^{m-1} \frac{m}{(m-1)^2} \frac{b^2}{a} k_1 \right ] \left(T-t\right)^{\frac{-p}{p-1}}\,,
\end{equation*}
\begin{equation*}
\underline{\gamma}(t):=C^{p-1} \left(T-t\right)^{\frac{-p}{p-1}}\,.
\end{equation*}

Let $p>m$. Conditions \eqref{eq7} and \eqref{eq8} imply \eqref{eq76} and \eqref{eq77}, whereas \eqref{eq8bis} and \eqref{eq8tris} imply \eqref{eq76b} and \eqref{eq77b}. Hence, by Propositions \ref{propsubsolution} and \ref{cpsub} the thesis follows in this case.

Let $p<m$. Conditions \eqref{eq9} and \eqref{eq10} imply \eqref{eq76} and \eqref{eq77}, whereas conditions  \eqref{eq8bis} and \eqref{eq8tris}  imply \eqref{eq76b} and \eqref{eq77b}. Hence, by Propositions \ref{propsubsolution} and \ref{cpsub}  the thesis follows in this case, too.

Finally, let $p=m$. Condition \eqref{eq9c} implies \eqref{eq76}, \eqref{eq77}, \eqref{eq76b} and \eqref{eq77b}. Hence, by Propositions \ref{propsubsolution} and \ref{cpsub}  the thesis follows in this case, too. The proof is complete.

\end{proof}


\section{Blow-up for any nontrivial initial datum: proofs}\label{bu}
\begin{proof}[Proof of Theorem \ref{teo4}]
Since $u_0\not\equiv 0$ and $u_0\in C(\mathbb R^N)$, there exist $\varepsilon>0, r_0>0$ and $x_0\in \mathbb R^N$ such that
 $$u_0(x) \ge \varepsilon, \quad \textrm{for all}\,\,\, x \in B_{r_0}(x_0).$$
Without loss of generality, we can assume that $x_0=0.$  Let  $\underline w$ be the subsolution of problem \eqref{problema} considered in Theorem \ref{teosubsolution} (with $a>0$ and $C>0$ properly fixed). We can find $T>0$ sufficiently big in such a way that
\begin{equation}
C\, T^{-\frac{1}{p-1}} \le \varepsilon, \quad  a\, T^{-\frac{m-p}{p-1}} \le \min\{r_0^b, r_0^2\}.
\label{condcorollario}
\end{equation}
From inequalities in \eqref{condcorollario}, we can deduce that
\begin{equation*}
\underline w(x,0) \le u_0(x) \quad \textrm{for any}\,\,\, x\in \R.
\end{equation*}
Hence, by Theorem \ref{teosubsolution} and the comparison principle, the thesis follows.
\end{proof}

Let us explain the strategy of the proof of Theorem \ref{teo3}. Let $u$ be a solution to problem \eqref{problema} and let $\underline w$ be the subsolution to problem \eqref{problema} given by Theorem \ref{teosubsolution}. We look for a subsolution $z$ to the equation
\begin{equation}
z_t= \frac{1}{\rho(x)} \Delta (z^m) \quad \text{in } \R^N\times (0,\infty)\,,
\label{equazionespeciale}
\end{equation}
such that
\begin{equation}\label{conddatiiniziali}
 z(x,0) \le u_0(x) \quad \text{for any }\,\, x\in \R^N\,,
 \end{equation}
 and
\begin{equation}
z(x,t_1) \ge \underline w(x,0)\quad  \text{for any}\,\,\, x\in \R^N \label{condtempot1}
\end{equation}
for $t_1>0$ and $T>0$ large enough.  Let $\tau >0$ be the maximal existence time of $u$. If $\tau\leq t_1$, then nothing has to be proved, and $u(x,t)$ blows-up at a certain time $S\in (0, t_1]$. Suppose that $\tau>t_1.$  Since $z$ is also a subsolution to problem \eqref{problema}, due to \eqref{conddatiiniziali} and the comparison principle,
\begin{equation}\label{eq103}
z(x,t) \le u(x,t) \quad \text{for any}\,\,\, (x,t)\in \R^N\times (0, \tau)\,.
\end{equation}
From \eqref{condtempot1} and \eqref{eq103},
$$u(x,t_1)\ge z(x,t_1) \ge \underline w(x,0) \quad \text{for any} \,\,\, x\in \R^N.$$
Thus $u(x, t+t_1)$ is a supersolution, whereas $\underline w(x, t)$ is a subsolution of problem
\[\begin{cases}
u_t = \frac 1{\rho} \Delta(u^m) + u^p & \textrm{in }\,\, \mathbb R^N\times (0, +\infty)\\
u(x, t_1) = \underline w(x, 0) & \textrm{in }\,\, \mathbb R^N\times \{0\}\,.
\end{cases}
\]
Hence by Theorem \ref{teosubsolution}, $u(x,t)$ blows-up in a finite time $S\in (t_1, t_1+T)$.

\bigskip

In order to construct a suitable family of subsolutions of equation \eqref{equazionespeciale}, let us consider two functions $\eta(t), \zeta(t) \in C^1([0, +\infty ); [0, +\infty ) )$ and two constants $C_1 > 0$, $a_1 > 0$. Define
\begin{equation}\label{eq100}
z(x,t)\equiv z(r(x),t):=
\begin{cases}
& \xi(x,t) \quad \text{in}\,\,\, [\R^N\setminus B_1(0)] \times (0,+\infty) \\
& \mu(x,t) \quad \text{in}\,\,\, B_1(0) \times (0,+\infty)\,,
\end{cases}
\end{equation}
where
\begin{equation}\label{eq101}
{\xi}(x,t)\equiv \xi(r(x),t):=C_1\zeta(t)\left [1-\frac{r^b}{a_1}\eta(t)\right]_{+}^{\frac{1}{m-1}}
\end{equation}
and
\begin{equation}\label{eq101b}
{\mu}(x,t)\equiv \xi(r(x),t):=C_1\zeta(t)\left [1-\frac{br^2+2-b}{2a_1}\eta(t)\right]_{+}^{\frac{1}{m-1}}\,.
\end{equation}
Let us set
$$
F(r,t):= 1-\frac{r^b}{a_1}\eta(t)\,, \quad G(r,t):= 1-\frac{br^2+2-b}{2a_1}\eta(t)
$$
and define
$$
D_1:=\left \{ (x,t) \in [\R^N \setminus B_1(0)] \times (0,+\infty)\, |\, 0< F(r,t) <1 \right \},
$$
$$
D_2:=\left \{ (x,t) \in B_1(0) \times (0,+\infty)\, |\, 0< G(r,t) <1 \right \}.
$$

Furthermore, for $\epsilon_0>0$ small enough, let
\begin{equation}\label{beta0}
\beta_0=\frac{b\,\dfrac{k_1}{k_2}}{(m-1)\left(N-2 \right)+bm}\,,
\end{equation}
\begin{equation}\label{alpha0}
\alpha_0:=\frac{1-\beta_0}{m-1}=\frac{N-2+\dfrac{b}{m-1}\left(m-\dfrac{k_1}{k_2}\right)}{(m-1)\left(N-2 \right)+bm}\,,
\end{equation}
\begin{equation}\label{beta0b}
\tilde \beta_0=\frac{ 2\dfrac{k_1}{k_2}-\epsilon_0}{N(m-1)+2}\,,
\end{equation}
\begin{equation}\label{alpha0b}
\tilde \alpha_0:=\frac{1-\tilde \beta_0}{m-1}=\frac{N(m-1)+2-2\dfrac{k_1}{k_2}+\epsilon_0}{(m-1)\left[N(m-1)+2 \right ]}\,,
\end{equation}
Observe that
\begin{equation}\label{betabound}
0\,<\,\beta_0\,<\,1, \quad\quad 0\,<\,\tilde \beta_0\,<\,1.
\end{equation}
Note that, if $\epsilon_0>0$ is small enough, then
\begin{equation}\label{betabound2}
0\,<\,\beta_0\,<\,\tilde\beta_0.
\end{equation}

\begin{proposition} Let assumption \eqref{hp}  be satisfied. Assume that \eqref{eq132} holds, for $\epsilon>0$ small enough.
Let
\begin{equation}
\bar{\beta}\in (0, \beta_0)\,,
\label{eq102b}
\end{equation}
\begin{equation}
\bar{\alpha}:=\frac{1-\bar\beta}{m-1}\,.
\label{eq102}
\end{equation}
Suppose that
\begin{equation}\label{eq138}
1<p<m+ \frac{\bar\beta}{\bar\alpha}\,.
\end{equation}
Let $T_1\in (0,\infty)$,
$$\zeta(t)=(T_1+t)^{-\bar{\alpha}}, \quad \eta(t)=(T_1+t)^{-\bar{\beta}}\,.$$
Then there exist $\omega_1:=\frac{C_1^{m-1}}{a_1}>0$, $t_1>0$ and $T>0$ such that $z$ defined in \eqref{eq100} is a  subsolution of equation \eqref{equazionespeciale} and satisfies \eqref{conddatiiniziali} and \eqref{condtempot1}.
\label{propsubsolutionspeciale}
\end{proposition}

\begin{proof}
We can argue as we have done to obtain \eqref{eq72}, in order to get
\begin{equation}\label{eq104}
\begin{aligned}
&\xi_t-\frac 1 {\rho} \Delta(\xi^m) \\
&\le C_1F^{\frac{1}{m-1}-1} \left \{F\left [\zeta ' + \zeta \frac{1}{m-1} \frac{\eta'}{\eta} + C_1^{m-1} \zeta^m \frac{m}{m-1} \frac{b}{a_1}\eta k_2 \left (N-2+ \frac{bm}{m-1} \right ) \right ] \right . \\
& \left . -\zeta \frac{1}{m-1} \frac{\eta'}{\eta} - C_1^{m-1} \zeta^m \frac{m}{(m-1)^2} \frac{b^2}{a_1}\eta k_1\right \}\quad \textrm{for all}\,\,\, (x,t)\in D_1\,.
\end{aligned}
\end{equation}
We now define
\begin{equation}
\begin{aligned}
& \sigma(t) := \zeta ' + \zeta \frac{1}{m-1} \frac{\eta'}{\eta} + C_1^{m-1} \zeta^m \frac{m}{m-1} \frac{b}{a}\eta k_2 \left (N-2+ \frac{bm}{m-1} \right ),\\
& \delta(t) := \zeta \frac{1}{m-1} \frac{\eta'}{\eta} + C_1^{m-1} \zeta^m \frac{m}{(m-1)^2} \frac{b^2}{a}\eta k_1.
\end{aligned}
\end{equation}
Hence, \eqref{eq104} becomes
\begin{equation}\label{eq105}
\xi_t-\frac 1 {\rho} \Delta(\xi^m) \leq C_1F^{\frac{1}{m-1}-1} \bar{\varphi}(F)\quad \textrm{in}\,\,\,D_1\,,
\end{equation}
where
\begin{equation}
\bar{\varphi}(F):=\sigma(t)F - \delta(t).
\end{equation}
Observe that $\xi$ is a subsolution to equation
\begin{equation}\label{eq106}
\xi_t-\frac 1 {\rho} \Delta(\xi^m) = 0 \quad \textrm{in}\,\,\, D_1\,,
\end{equation}
whenever, for any $t>0$
\[
\bar{\varphi}(F) \le 0,
\]
that is
\begin{equation}
\begin{cases}
\sigma(t) >0 \\
\delta(t) >0 \\
\sigma(t) -\delta(t) \le 0.
\end{cases}\quad \textrm{for any}\,\,\, t>0\,
\label{sistemaspeciale}
\end{equation}
By using the very definition of $\zeta$ and $\eta$, we get
\[\sigma(t)=-\bar{\alpha}(T_1+t)^{-\bar{\alpha}-1} - \frac{\bar{\beta}}{m-1}(T_1+t)^{-\bar{\alpha}-1}+\frac{C_1^{m-1}}{a_1}k_2\frac{m}{m-1}b\left(N-2+\frac{bm}{m-1}\right)(T_1+t)^{-\bar{\alpha}m-\bar{\beta}},\]
\[\delta(t)=-\frac{\bar{\beta}}{m-1}(T_1+t)^{-\bar{\alpha}-1}+\frac{C_1^{m-1}}{a_1}k_1\frac{m}{(m-1)^2}b^2(T_1+t)^{-\bar{\alpha}m-\bar{\beta}}\,.\]
By \eqref{betabound}, \eqref{eq102b} and \eqref{eq102},
\begin{equation}
0<\bar{\beta}<1, \quad \bar{\alpha}>0\,.
\end{equation}
Due to \eqref{eq102}, \eqref{sistemaspeciale} becomes
\begin{equation}\label{eq111}
\begin{cases}
-1+\dfrac{C_1^{m-1}}{a_1}k_2\,m\,b\,\left (N-2+\dfrac{bm}{m-1} \right ) >0,\\
-\bar{\beta}+\dfrac{C_1^{m-1}}{a_1}\,k_1\,\dfrac{m}{m-1}\,b^2 >0, \\
\bar{\beta}-1+\dfrac{C_1^{m-1}}{a_1}\,b\,m \,\left [ k_2\left (N-2+\dfrac{bm}{m-1}\right )-k_1\dfrac{b}{m-1} \right ] \le 0\,,
\end{cases}
\end{equation}
which reduces to
\begin{equation}\label{eq114}
\dfrac{C_1^{m-1}}{a_1} \ge \max \left \{ \dfrac{1}{b\,m\,k_2\left (N-2+\dfrac{bm}{m-1}\right)}, \dfrac{\bar{\beta}(m-1)}{b^2\,m\,k_1} \right \}\,,
\end{equation}
\begin{equation}\label{eq117}
\dfrac{C_1^{m-1}}{a_1} \le \dfrac{1-\bar{\beta}}{b\,m \left [k_2\left(N-2+\dfrac{bm}{m-1}\right)-k_1\dfrac{b}{m-1} \right ]}\,.
\end{equation}
If \eqref{eq114} and \eqref{eq117} are verified, then $\xi$ is a subsolution to equation \eqref{eq106}. We now show that it is possible to find $\omega_1:=\frac{C_1^{m-1}}{a_1}$ such that \eqref{eq114} \eqref{eq117} hold. Such $\omega_1$ can be selected, if
\begin{equation}\label{eq2000}
\dfrac{1}{b\,m\,k_2\left (N-2+\dfrac{bm}{m-1}\right)}\,\,<\,\,\dfrac{1-\bar{\beta}}{b\,m \left [k_2\left(N-2+\dfrac{bm}{m-1}\right)-k_1\dfrac{b}{m-1} \right ]}\,,
\end{equation}
and
\begin{equation}\label{eq2001}
\dfrac{\bar{\beta}(m-1)}{b^2\,m\,k_1} \,\, < \,\, \dfrac{1-\bar{\beta}}{b\,m \left [k_2\left(N-2+\dfrac{bm}{m-1}\right)-k_1\dfrac{b}{m-1} \right ]}\,.
\end{equation}
Conditions \eqref{eq2000} and \eqref{eq2001} are satisfied, if
\begin{equation}\label{eq2002}
\bar \beta < \beta_0\,.
\end{equation}
Finally, condition \eqref{eq2002} is guaranteed by hypothesis \eqref{eq102b}. Moreover, by Lemma \ref{lemext}, $\xi$ is a subsolution to equation
\begin{equation}\label{eq107}
\xi_t - \frac1{\rho(x)}\Delta \xi^m = 0 \quad \textrm{in}\,\,\, [\mathbb R^N\setminus B_1(0)]\times (0, T)\,.
\end{equation}
in the sense of Definition \ref{soldom}. We can argue as we have done to obtain \eqref{eq93}, in order to get
\begin{equation}\label{eq108}
\begin{aligned}
&\mu_t- \frac{1}{\rho}\Delta(\mu^m)   \\
&\le C_1 G^{\frac{1}{m-1}-1}\left\{G\left[\zeta' + \frac{\zeta}{m-1}\frac{\eta'}{\eta}+ b\,k_2\,\frac{m}{m-1}\,\frac{C_1^{m-1}}{a_1}\,\zeta^m \,\eta\left (N+\frac{2}{m-1}\right )\right]\right .\\
&\left .-\frac{\zeta}{m-1}\frac{\eta'}{\eta} - 2\,k_1\,b\,\frac{C_1^{m-1}}{a_1} \frac{m}{(m-1)^2}\zeta^m\,\eta \,+\, (2-b)\,k_2\,b\,\frac{C_1^{m-1}}{a_1^2} \frac{m}{(m-1)^2} \,\zeta^m\,\eta^2 \right \}\, \\
& \quad \text{for any}\,\,\,(x,t)\in D_2\,.
\end{aligned}
\end{equation}
We now define
$$
\begin{aligned}
&\underline{\sigma}_0(t) := \zeta ' + \zeta \frac{1}{m-1} \frac{\eta'}{\eta} + b\,k_2\,\frac{C_1^{m-1}}{a_1} \zeta^m \frac{m}{m-1} \left( N+ \frac 2{m-1}\right)\,\eta,\\
&\underline{\delta}_0(t) := -\frac{\zeta}{m-1}\frac{\eta'}{\eta} +2\,k_1\,b\,\frac{C_1^{m-1}}{a_1} \frac{m}{(m-1)^2}\zeta^m\,\eta \,-\, (2-b)\,k_2\,b\,\frac{C_1^{m-1}}{a_1^2} \frac{m}{(m-1)^2} \,\zeta^m\,\eta^2\,.
\end{aligned}
$$
Hence, \eqref{eq108} becomes,
\begin{equation}\label{eq108b}
\mu_t- \frac{1}{\rho}\Delta(\mu^m) \le C_1 \,G^{\frac{1}{m-1}-1}\, \phi(G) \,\,\, \text{in}\,\,\, D_2\,,
\end{equation}
where
$$
\phi(G):=\underline{\sigma}_0(t) G - \underline{\delta}_0(t)\,.
$$
By arguing as above, we can infer that
\begin{equation}\label{eq109}
\mu_t  - \frac{1}{\rho}\Delta(\mu^m) \leq 0 \quad \textrm{in}\,\,\, D_2\,,
\end{equation}
provided that
\begin{equation}\label{eq111b}
\begin{cases}
\sigma_0(t) >0 \\
\delta_0(t) >0 \\
\sigma_0(t) -\delta_0(t) \le 0\,.
\end{cases}\quad \textrm{for any}\,\,\, t\in (0, T_1)\,
\end{equation}
By using the very definition of $\zeta$ and $\eta$, \eqref{eq111b} becomes
\begin{equation}\label{eq113}
\begin{aligned}
&- 1+b\,k_2\, \dfrac{C_1^{m-1}}{a_1}\, m \,\left (N+\dfrac{2}{m-1} \right ) >0,\\
&-\bar{\beta}\,+\,2\, b\,k_1\,\dfrac{C_1^{m-1}}{a_1}\dfrac{m}{m-1} \,-\,(2-b)\,b\,k_2\,\frac{C_1^{m-1}}{a_1^2}\frac{m}{m-1}(T_1+t)^{-\bar\beta}>0, \\
&\bar{\beta}\,-\,1\,+\,b\,k_2\,m\,\dfrac{C_1^{m-1}}{a_1} \,N\,+\frac{2}{m-1}\left(1-\frac{k_1}{k_2}\right)\,+\,(2-b)\,k_2\,b\,\dfrac{C_1^{m-1}}{a_1^2}\dfrac m{m-1} (T_1+t)^{-\bar\beta} \le 0\,,
\end{aligned}
\end{equation}
which reduces to
\begin{equation}\label{eq115}
\dfrac{C_1^{m-1}}{a_1} > \max \left \{ \dfrac{1}{b\, m\,k_2\left (N+\dfrac{2}{m-1}\right)}\,\,,\,\, \dfrac{\bar{\beta}(m-1)}{ b\, m\,k_2\left[2\dfrac{k_1}{k_2}-\dfrac{2-b}{a_1}(T_1+t)^{-\bar\beta}\right] } \right \}\,,
\end{equation}
\begin{equation}\label{eq116}
\dfrac{C_1^{m-1}}{a_1} \leq  \dfrac{1-\bar{\beta}}{ b\, m\,k_2\left[N+\dfrac{2}{m-1}\left(1-\dfrac{k_1}{k_2}\right)+ \dfrac{2-b}{a_1}\dfrac{(T_1+t)^{-\bar\beta}}{m-1}\right]}\,.
\end{equation}
If \eqref{eq115} and \eqref{eq116} are verified then $\mu$ is a subsolution to equation
$$
\mu_t-\frac{1}{\rho}\Delta \mu^m=0 \quad \text{in}\,\,\,D_2\,.
$$
In order to find $\omega_1=\frac{C_1^{m-1}}{a_1}$ satisfying \eqref{eq115} and \eqref{eq116}, we need
\begin{equation}\label{eq2000b}
\dfrac{1}{bmk_2\left (N+\dfrac{2}{m-1}\right)}\,<\,\dfrac{1-\bar{\beta}}{bmk_2 \left[N+\dfrac{2}{m-1}\left(1-\dfrac{k_1}{k_2}\right)+ \dfrac{2-b}{a_1}\dfrac{(T_1+t)^{-\bar\beta}}{m-1}\right]}\,,
\end{equation}
and
\begin{equation}\label{eq2001b}
\dfrac{\bar{\beta}(m-1)}{bmk_2\left[2\dfrac{k_1}{k_2}- \dfrac{2-b}{a_1}(T_1+t)^{-\bar\beta}\right]} \, < \, \dfrac{1-\bar{\beta}}{bmk_2 \left[N+\dfrac{2}{m-1}\left(1-\dfrac{k_1}{k_2}\right)+ \dfrac{2-b}{a_1}\dfrac{(T_1+t)^{-\bar\beta}}{m-1}\right]}\,.
\end{equation}
Now we choose in \eqref{eq132} $\epsilon=\epsilon(a_1,T_1)>0$ so that
\begin{equation}\label{eq20f}
\frac{\epsilon}{a_1}T_1^{-\bar{\beta}}\,\le\, \epsilon_0\,,
\end{equation}
with $\epsilon_0$ used in \eqref{beta0b} and \eqref{alpha0b} to be appropriately fixed.
By \eqref{eq132}, \eqref{beps} and \eqref{eq20f},
$$
\frac{2-b}{a_1}\left(T_1+t\right)^{-\bar\beta}\,\,<\,\,\frac{\epsilon}{a_1}T_1^{-\bar{\beta}}\,\le\, \epsilon_0.
$$
So, conditions \eqref{eq2000b} and \eqref{eq2001b} are fulfilled, if
\begin{equation}\label{eq2000c}
\dfrac{1}{b\,m\,k_2\left (N+\dfrac{2}{m-1}\right)}\,<\,\dfrac{1-\bar{\beta}}{b\,m\,k_2 \left[N+\dfrac{2}{m-1}\left(1-\dfrac{k_1}{k_2}\right)+ \dfrac{\epsilon_0}{m-1}\right]}\,,
\end{equation}
and
\begin{equation}\label{eq2001c}
\dfrac{\bar{\beta}(m-1)}{b\,m\,k_2\left[2\dfrac{k_1}{k_2}- \epsilon\right]} \, < \, \dfrac{1-\bar{\beta}}{b\,m\,k_2 \left[N+\dfrac{2}{m-1}\left(1-\dfrac{k_1}{k_2}\right)+ \dfrac{\epsilon_0}{m-1}\right]}\,.
\end{equation}
Finally, conditions \eqref{eq2000c} and \eqref{eq2001c} are satisfied, if
\begin{equation}\label{eq2002b}
\bar \beta < \tilde \beta_0\,,
\end{equation}
provided that $\epsilon_0>0$ is small enough. Observe that \eqref{eq2002b} is guaranteed due to hypothesis  \eqref{betabound2} and \eqref{eq102b}. Moreover, since $\mu^m\in C^1(B_1(0)\times (0, T_1))$, by Lemma \ref{lemext}, $\mu$ is a subsolution to
\begin{equation}\label{eq110}
\mu_t - \frac{1}{\rho}\Delta(\mu^m) = 0 \quad \textrm{in}\,\,\, B_1(0)\times (0, T_1)\,,
\end{equation}
in the sense of Definition \ref{soldom}. Hence $z$ is a subsolution of equation \eqref{equazionespeciale}.

\bigskip

Since $u_0\not\equiv 0$ and $u_0\in C(\R^N)$, there exist $r_0>0$ and $\varepsilon>0$ such that
$$u_0(x) > \varepsilon \quad \textrm{in}\,\,\, B_{r_0}(0).$$
Hence, if
\begin{equation}\label{eq200}
\operatorname{supp} \, z(\cdot, 0) \subset B_{r_0}(0),
\end{equation}
and
\begin{equation}\label{eq201}
z(x,0) \le \varepsilon \quad \text{in } B_{r_0}(0),
\end{equation}
then \eqref{conddatiiniziali} follows. Moreover, if
\begin{equation}\label{eq202}
\operatorname{supp}\,\, \underline w(\cdot,0) \subset \operatorname{supp}\,\, z(\cdot , t_1)\,,
\end{equation}
and
\begin{equation}\label{eq203}
\underline w(x,0) \le z(x, t_1)\,\,\, \text{for all}\,\, x \in \R^N,
\end{equation}
then \eqref{condtempot1} follows.

\bigskip
We first verify that $z$ satisfies condition \eqref{eq200} and \eqref{eq201}. If we require that
\begin{equation}
a_1 \,T_1^{\bar \beta} \le \,\frac{r_0^2}{2}\,.
\label{sistemadatiiniziali}
\end{equation}
then
$$
\operatorname{supp} \, z(\cdot, 0) \cap B_1(0) \subset B_{r_0}(0)\,,
$$
and
$$
\operatorname{supp} \, z(\cdot, 0) \cap [\R^N \setminus B_1(0)] \subset B_{r_0}(0)\,,
$$
therefore \eqref{eq200} holds. Moreover, if
\begin{equation}\label{eq204}
(a_1 \,\omega)^{\frac{1}{m-1}}\, \le \,\varepsilon \,T_1^{\bar{\alpha}},
\end{equation}
then \eqref{eq201} holds.
Obviously, for any $T_1>0$ we can choose $a_1=a_1(T_1)>0$ such that  \eqref{sistemadatiiniziali} and \eqref{eq204} are valid.
On the other hand,
$$
\operatorname{supp} \, \underline w(\cdot, 0) \cap B_1(0) \subset \operatorname{supp} \, z(\cdot, t_1) \cap B_{1}(0)\,,
$$
and if
\begin{equation}\label{eq136}
a_1\,(T_1+t_1)^{\bar{\beta}} \ge a\,T^{\frac{p-m}{p-1}}
\end{equation}
then,
$$
\operatorname{supp} \, \underline w(\cdot, 0) \cap [\R^N \setminus B_1(0)] \subset \operatorname{supp} \, z(\cdot, t_1) \cap [\R^N \setminus B_1(0)].
$$
Hence, \eqref{eq202} holds. If
\begin{equation}\label{eq136b}
C_1\,(T_1+t_1)^{-\bar{\alpha}} \ge C\,T^{-\frac{1}{p-1}},
\end{equation}
then \eqref{eq203} holds. If we choose the equality in \eqref{eq136b},
$$T_1+t_1=\left ( \frac{C}{C_1} \right ) ^{-\frac{1}{\bar{\alpha}}} T^{\frac{1}{(p-1)\bar{\alpha}}},$$ then \eqref{eq136} becomes
$$
\left ( \frac{C}{C_1} \right ) ^{-\frac{\bar{\beta}}{\bar{\alpha}}}a_1 \,T^{\frac{\bar{\beta}}{\bar{\alpha}}\frac{1}{(p-1)}} \ge a \,T^{\frac{p-m}{p-1}}\,.
$$
The latter holds, if
\begin{equation}
T^{\frac{p-m-\frac{\bar{\beta}}{\bar{\alpha}}}{p-1}} \le \left ( \frac{C}{C_1} \right ) ^{-\frac{\bar{\beta}}{\bar{\alpha}}}\frac{a_1}{a}\,.
\label{condsuT}
\end{equation}
Condition \eqref{condsuT} is satisfied thanks to \eqref{eq138}, for $T>0$ sufficiently large.
This completes the proof.
\end{proof}

\bigskip
\begin{proof}[Proof of Theorem \ref{teo3}]
Let $\tau>0$ be the maximal existence time of $u$. If $\tau\leq t_1$, then nothing has to be showed, and $u$ blows-up at a certain time $S\in (0, t_1]$. Suppose $\tau>t_1$.
Let us consider the subsolution  $z$  of equation \eqref{equazionespeciale} as defined in \eqref{eq100}. Since $p<\underline p$, we can find $\bar \beta$ (and so $\bar \alpha$) such that \eqref{eq102b}, \eqref{eq102} and \eqref{eq138} hold. By Proposition \ref{propsubsolutionspeciale}, $z$ satisfies \eqref{conddatiiniziali} and \eqref{condtempot1}. Thanks to condition \eqref{conddatiiniziali} and the comparison principle, we have \eqref{eq103}.
From \eqref{condtempot1} and \eqref{eq103},
$$
u(x,t_1)\ge z(x,t_1) \ge \underline w(x,0) \quad \text{for any} \,\,\, x\in \R^N.
$$
Thus $u(x, t+t_1)$ is a supersolution, whereas $\underline w(x, t)$ is a subsolution of problem
\[\begin{cases}
u_t = \frac 1{\rho} \Delta(u^m) + u^p & \textrm{in }\,\, \mathbb R^N\times (0, +\infty)\\
u = \underline w & \textrm{in }\,\, \mathbb R^N\times \{0\}\,.
\end{cases}
\]
Hence by Theorem \ref{teosubsolution}, $u(x,t)$ blows-up in a finite time $S\in (t_1, t_1+T)$. This completes the proof.
\end{proof}

\end{document}